\documentclass[AMS,Times1COL]{NJDv5}

\usepackage[utf8]{inputenc}
\usepackage{graphicx,xcolor,float}
\usepackage{amsmath,amssymb,bm,pgfplots}
\usepackage[sort,compress]{cite}
\usepackage[capitalise,nameinlink]{cleveref}
\pgfplotsset{compat=1.17}

\crefname{subsection}{Subsection}{Subsections}
\crefformat{equation}{(#2#1#3)}
\crefformat{align}{(#2#1#3)}
\crefformat{enumi}{(#2#1#3)}
\crefname{figure}{Figure}{Figures}

\allowdisplaybreaks

\renewcommand{\div}{\textup{div}\mspace{2mu}}
\newcommand{\dd}{\textup{d}}

\def\dx{\,\textup{d}x}

\def\RR{\mathbb{R}}
\def\eps{\varepsilon}
\def\dt#1{#1_t}
\def\dtt#1{#1_{tt}}

\newcommand{\ddt}{\frac{\textup{d}}{\textup{d}t}}

\newcommand{\ds}{\,\textup{d}s}

\renewcommand{\cal}{\mathcal}

\usetikzlibrary{spy}

\newcolumntype{M}[1]{>{\centering\arraybackslash}m{#1}}

\articletype{}

\received{}
\revised{}
\accepted{}
\journal{}
\volume{}
\copyyear{}
\startpage{}
\usepackage{comment,float,graphicx}
\renewcommand{\cite}{\citep}

\begin{document}

\setcounter{page}{1}

\title{Well-posedness, long-time behavior, and discretization of some models of nonlinear acoustics in velocity--enthalpy formulation}
\author[1,2]{Herbert Egger}
\author[1]{Marvin Fritz}
\authormark{EGGER \textsc{and} FRITZ}
\titlemark{Well-posedness and stability of some models of nonlinear acoustics}
\address[1]{\orgdiv{Computational Methods for PDEs}, \orgname{Johann Radon Institute for Computational and Applied Mathematics}, \orgaddress{\state{Linz}, \country{Austria}}}
\address[2]{\orgdiv{Institute of Numerical Mathematics}, \orgname{Johannes Kepler University}, \orgaddress{\state{Linz}, \country{Austria}}}
\def\myalpha{\eta}
\corres{Marvin Fritz, Johann Radon Institute for Computational and Applied Mathematics, Linz, Austria. \email{marvin.fritz@ricam.oeaw.ac.at}}

\abstract[Abstract]{%
We study a class of models for nonlinear acoustics, including the well-known Westervelt and Kuznetsov equations, as well as a model of Rasmussen that can be seen as a thermodynamically consistent modification of the latter. 
Using linearization, energy estimates, and fixed-point arguments, we establish the existence and uniqueness of solutions that, for sufficiently small data, are global in time and converge exponentially fast to equilibrium. 
In contrast to previous work, our analysis is based on a velocity-enthalpy formulation of the problem, whose weak form reveals the underlying port-Hamiltonian structure. Moreover, the weak form of the problem is particularly well-suited for a structure-preserving discretization.
This is demonstrated in numerical tests, which also highlight typical characteristics of the models under consideration.}


\jnlcitation{\cname{%
\author{Egger H}, and
\author{Fritz M}}.
\ctitle{Global well-posedness and long-time stability of some models of nonlinear acoustics.}}
\maketitle
\thispagestyle{empty}

\counterwithin{lemma}{section}
\counterwithin{equation}{section}

\section{Introduction} 
\label{sec:intro}
Models for nonlinear acoustics are usually derived via certain approximations of the compressible Euler equations; see  \cite{hamilton1997nonlinear,kaltenbacher2007numerical} for instance. As a general model governing the velocity potential $\psi$, we consider the equation
\begin{align} \label{eq:model} 
(1- a \dt \psi) \dtt \psi - c_0^2 \Delta \psi = b \Delta \dt \psi +  c \nabla \psi \cdot \nabla \dt \psi + d \dt \psi \Delta \psi.
\end{align}
Here, $c_0$ denotes the speed of sound in the unperturbed medium, $b$ is the coefficient of thermoviscous damping, and $a,c,d$ are further model parameters describing the nonlinear effects in the medium. 
Several well-established models of nonlinear acoustics are covered by \eqref{eq:model}, e.g., the Westervelt equation ($a=(\gamma+1)/c_0^2$, $c=d=0$), the Kuznetsov equation ($a=(\gamma-1)/c_0^2$, $c=2$, $d=0$), or its modification by Rasmussen ($a=(\gamma-2)/c_0^2$, $c=2$, $d=1$);
see \cite{jordan2016survey,rasmussen2009thermoviscous} for a comparison and detailed derivations.

\medskip 
\noindent 
\textbf{Motivation and related work.}
Due to its seemingly superior accuracy~\cite{christov2007modeling}, the Kuznetsov equation has been widely used in the engineering community; see~\cite{abramov2019high,dreyer2000investigations,hoffelner2001finite,kaltenbacher2002use} and the references given there.
In contrast to the underlying compressible Euler system and the Westervelt equation, it,
however, lacks thermodynamic consistency, that is, nonphysical energy production cannot be ruled out rigorously. 
This was the motivation for Rasmussen et al.~to propose a modification that has similar accuracy, but provably encodes the passivity of the system; see ~\cite{rasmussen2009thermoviscous,rasmussen2010analytical,rassmusen2011interacting} for details. 
In our numerical tests, we observe that this model also has somewhat superior stability properties.
The mathematical analysis of the Westervelt and Kuznetsov equations has been comprehensively addressed by Kaltenbacher et al.~\cite{kaltenbacher2011well,kaltenbacher2012analysis,kaltenbacher2011west,kaltenbacher2018fundamental}, who employed the velocity potential formulation. 
Further analytical results based on maximal $L^p$-regularity were obtained in \cite{meyer2011optimal,meyer2013global} for analytical results , and \cite{tani2017mathematical} addressed the Cauchy problem for a related model.
The well-posedness of other models of nonlinear acoustics, such as the Blackstock equation, the third-order Moore--Gibson--Thompson equation, and models with nonlocal-in-time dissipation were considered in \cite{fritz2018well,kaltenbacher2024kuznetsov,marchand2012abstract,tu2023well}.
In \cite{dorich2024robust,nikolic2019priori}, an apriori error analysis for appropriate discretization schemes was presented, \cite{meliani2024mixed} explored a second-order mixed formulation,  \cite{kaltenbacher2024first} studied a different first-order system, \cite{kaltenbacher2021convergence,kaltenbacher2015efficient} considered the convergence of time discretization methods, and \cite{nikolic2023asymptotic} investigated asymptotic-preserving finite element methods for the Westervelt equation. 
\clearpage

\noindent 
\textbf{Scope and main contributions.}
In this paper, we investigate the velocity-enthalpy formulation of \eqref{eq:model}, see Section~\ref{sec:prelim}, and study the global existence and uniqueness of solutions, as well as their long-time behavior.
For parameters $c=2d$, which includes the Westervelt and Rasmussen equations, the weak form of this model reveals an underling port-Hamiltonian structure, thus guaranteeing passivity and thermodynamic consistency of the system.
In addition, the weak form of the velocity--enthalpy formulation turns out to be very well suited for a structure-preserving discretization which inherits the elementary dissipation mechanism and basic stability properties from the continuous problem. Some preliminary computational results will be presented for illustration of these findings.

\medskip 
\noindent 
\textbf{Outline.}
The remainder of the paper is organized as follows: 
In \cref{sec:prelim}, we briefly derive and formally state the velocity-enthalpy formulation of \eqref{eq:model}, introduce our notation, and present our main theoretical results, whose proof is elaborated in the following sections. 
A linearized problem is investigated in \cref{sec:lin} and local existence of a unique solution is proven in \cref{sec:local}. The global existence and decay of solutions are finally established in \cref{sec:global}.
In \cref{sec:num}, we discuss the structure-preserving discretization of the problem through mixed finite elements and an implicit time-stepping scheme. We further present some numerical tests which illustrate the passivity and the exponential decay of solutions on the discrete level.

\section{Preliminaries and main results}
\label{sec:prelim}

Let us start with giving a complete definition of the model problem to be considered in the rest of the manuscript. 
We assume that $c_0^2=1$ in the sequel, which can always be achieved by appropriate scaling of space and time. We further introduce the velocity $v=-\nabla \psi$ and the total specific enthalpy $h=\dt \psi$. Then \eqref{eq:model} can be rewritten equivalently as 
\begin{alignat}{2}
\dt v + \nabla h &= 0,  \qquad && \text{in } \Omega, \ t \ge 0, \label{eq:sys1}\\
(1-a h) \, \dt h + \div v &= b \Delta h - c \, v \cdot \nabla h - d \, h \, \div v \qquad && \text{in } \Omega, \ t\ge 0. \label{eq:sys2} 
\intertext{The system is complemented by appropriate boundary and initial conditions. For ease of notation, we choose}
h &= 0 \qquad && \text{on } \partial\Omega, \ t\ge 0, \label{eq:sys3}  \\
h(0) &= h_0, \qquad v(0) = v_0 \qquad && \text{in } \Omega. \label{eq:sys4}
\end{alignat}
Other types of boundary conditions could be treated with minor modification of our arguments below. 
Throughout the manuscript, we make the following assumptions about the problem data required for our analysis.
\begin{assumption} \label{ass:1}
$\Omega \subset \RR^n$, $n \le 3$, is a bounded domain with $C^{3,1}$-boundary. In addition, $a,c,d \ge 0$ and $b>0$ are constants.
\end{assumption}
Suitable conditions on the initial values are stated explicitly in the statement of our results below. 
The smoothness assumption on the domain is used to employ regularity theorems for the Poisson problem and in turn esatlish regularity of solution to the nonlinear acoustics equations; see e.g. \cite[Theorem 20.4]{Wloka} and \cite[Ch~6.3]{evans2010partial}. 
Throughout the text, we use standard notation for function spaces and norms; see, for instance, \cite{evans2010partial,Wloka}.  
For ease of notation, we will use 
$$
(f,g)_\Omega=\int_\Omega f(x)g(x) \dx
$$ 
as abbreviation for the the scalar product of $L^2(\Omega)$. Further notation will be introduced where required.

\subsubsection*{Weak formulation and energy dissipation}

As a first step of our analysis, we present a weak characterization of solutions, which allows us to establish the basic power balance of the problem and which serves as the starting point for the design of structure-preserving numerical schemes later on. 
\begin{proposition} \label{pro:weak}
Let $(v,h)$ be a sufficiently smooth solution of \eqref{eq:sys1}--\eqref{eq:sys4} on $[0,T]$ for some $T>0$. Then 
\begin{align}
(\dt v,w)_\Omega + (\nabla h,w)_\Omega &= 0 \label{eq:weak1}\\
((1-a h) \dt h,q)_\Omega - (v,\nabla q)_\Omega &= -\, b (\nabla h, \nabla q)_\Omega - (c-d) (v \cdot \nabla h, q)_\Omega + d (h, v \cdot \nabla q)_\Omega \label{eq:weak2}
\end{align}
for all $w \in L^2(\Omega)^n$ and $q \in H_0^1(\Omega)$, and all $0 \le t \le T$. Moreover,
\begin{align} \label{eq:diss}
\ddt \mathcal{H}(v,h) &= -b \|\nabla h\|^2_{L^2} + (2d-c) (v \cdot \nabla h, h)_\Omega,
\end{align}
where $\mathcal{H}(v,h) := \int_\Omega \tfrac{1}{2} |v|^2 + \tfrac{1}{2} h^2 - \tfrac{a}{3} h^3 \, \textup{d}x$ is the Hamiltonian, i.e., the total energy of the system.
For $c=2d$, the energy thus decreases monotonically over time. 
\end{proposition}
\begin{proof}
We note that the last term in \eqref{eq:sys2} can be rewritten as 
\begin{align*}
-\, d \, h \, \div v = d \, v \cdot \nabla h - d \, \div (h \, v).
\end{align*}
The variational identities \eqref{eq:weak1}--\eqref{eq:weak2} are then obtained in the usual manner, that is, by multiplying the equations \eqref{eq:sys1}--\eqref{eq:sys2} with appropriate test functions and integration over the domain $\Omega$. For the second term on the left-hand side of \eqref{eq:weak2}, we used integration-by-parts. The corresponding boundary term vanishes due to the conditions on the test function $q$. 
By formal differentiation of the energy functional, we see that 
\begin{align}
\frac{\dd}{\dd t} \mathcal{H}(v,h)
&= (v_t, v)_\Omega + ((1-a h) h_t, h)_\Omega.
\end{align}
The power balance \eqref{eq:diss} then follows immediately from \eqref{eq:weak1}--\eqref{eq:weak2} by testing with $(w,q)=(v,h)$. 
\end{proof}

\begin{remark}
Let us emphasize that the power balance \eqref{eq:diss} is a direct consequence of the particular structure of the variational identities~\eqref{eq:weak1}--\eqref{eq:weak2}, which can be preserved under Galerkin projection. 
We will use this fact in Section~\ref{sec:num} below. 
Let us further note that the energy functional $\mathcal{H}(v,h)$ is convex provided that $\|h\|_{L^\infty}$ is sufficiently small. 
\end{remark}

\subsubsection*{Well-posedness and exponential decay}

The natural energy $\mathcal{H}(v,h)$ of the problem is not strong enough to provide uniform bounds for $h$ in $L^\infty$ in space-time, which are required to demonstrate the well-posedness of the problem. 
For our analysis, we therefore 
utilize the stronger energy 
\begin{equation} \label{Eq:EnergyMath} 
\mathcal{E}(v,h) :=\|h\|_{H^2}^2 + \|v\|_{H^1_\div}^2 := \|h(t)\|_{H^2}^2 + \|v(t)\|_{H^1}^2 + \|\div v(t)\|_{H^1}^2.
\end{equation} 
Using standard convention, we write $H^\ell_\div(\Omega):=\{v \in H^\ell(\Omega)^n : \div v \in H^\ell(\Omega)\}$, $\ell\in \mathbb{N}$,
for the space of regular velocity fields and denote by $\|v\|_{H^\ell_\div} = (\|v\|_{H^\ell}^2 + \|\div v\|_{H^\ell}^2)^{1/2}$ its natural norm. 
We search for a solution $(v,h)$ in the space $\mathcal{V}_T \times \mathcal{H}_T$ with
\begin{align*} 
\mathcal{V}_T&= \{ v \in  C([0,T];H^1_\div(\Omega))  \cap C^1([0,T];H^1(\Omega)^n) : \div v = 0 \text{ on } \partial\Omega\},   \\ 
\mathcal{H}_T&= \{h \in L^2(0,T;H^3(\Omega))\cap C([0,T];H^2(\Omega)) \cap C^1([0,T];L^2(\Omega)) : h = \Delta h = 0 \text{ on } \partial\Omega\}.
\end{align*}
Here $C([0,T];X)$ and $L^p(0,T;X)$ denote the spaces of functions $f:[0,T] \to X$ of time with values in some normed space~$X$; see~\cite{Wloka} for details and properties of these spaces. 
If not stated otherwise, the norm for the intersection of two Banach spaces is defined by $\|\cdot\|_{X\cap Y} := \max\{\|\cdot\|_X,\|\cdot\|_Y\}$.
We are now in a position to state our main results.
\begin{theorem}[Well-posedness and energy decay] \label{thm:main} $ $ \\
Let Assumption~\ref{ass:1} hold. Then for any $T>0$ there exists an $\eps(T)>0$ such that for all initial values 
\begin{align*}
(v_0,h_0) \in H^1_\div(\Omega) \times H^2(\Omega)  \quad \text{with} \quad 
\div v_0 = h_0 = 0 \quad \text{on } \partial\Omega
\quad \text{and} \quad 
\mathcal{E}(v_0,h_0) \le \eps(T),
\end{align*}
there exists a unique (local) solution $(v,h) \in \mathcal{V}_T \times \mathcal{H}_T$ of problem \eqref{eq:sys1}--\eqref{eq:sys4}. 
Furthermore, if $\mathcal{E}(v_0,h_0)\le \eps_\infty$ with $\eps_\infty>0$  sufficiently small, then the solution exists (globally) for all $t>0$ and there exist constants $C_0,C_1>0$ such that
\begin{align}
\|h(t)\|_{H^2}^2 + \|\div v(t)\|_{H^1}^2 \le C_0 \, e^{-C_1 t} \qquad \forall t \ge 0.
\end{align}
\end{theorem}
The detailed proof of these assertions is presented in the following three sections. 
For orientation of the reader, let us briefly mention the main ingredients for the analysis already here: 
The smallness condition on the initial data allows to guarantee that the coefficient $(1-\myalpha h)$ in front of the time derivative of \eqref{eq:sys2} remains bounded and uniformly positive. 
The existence of a unique solution local in time can then be established using Galerkin approximations, energy estimates, and a fixed-point argument. 
The latter is based on a linearization of the problem, discussed in Section~\ref{sec:lin}, while existence of a fixed-point is established in Section~\ref{sec:local}.
A careful derivation of additional energy estimates finally allows us to derive global bounds on the solution and the decay estimate follow via the Haraux--Lagnese inequality; see Section~\ref{sec:global} for details.

\section{A linearized problem}
\label{sec:lin}

In a first step of our analysis, we study the linear problem that arises from replacing $h$ in some of the terms in \eqref{eq:sys1}--\eqref{eq:sys2} by a known function $\myalpha$. We additionally add a source function $f$ which allows us to use the results in various ways later on.
Throughout this section, we thus consider the linear system 
\begin{alignat}{2}
v_t  +\nabla h &=0 \qquad  && \text{in } \Omega, \ t > 0, \label{eq:lin1}\\
h_t +\div v &= b\Delta h + a\myalpha h_t -  c\nabla \myalpha \cdot v  - d \myalpha \, \div v + f \qquad && \text{in } \Omega, \ t> 0,\label{eq:lin2}
\end{alignat}
which is again complemented by the initial and boundary conditions \eqref{eq:sys3}--\eqref{eq:sys4}. 
It is not difficult to see that sufficiently smooth solutions of this problem also satisfy the variational identities
\begin{align} 
(v_t,u)_\Omega &= -\,(\nabla h, u)_\Omega  \label{eq:lin1w}\\
(h_t,g)_\Omega + b (\nabla h,\nabla g)_\Omega  &=  (v,\nabla g)_\Omega +a(\myalpha  h_t,g)_\Omega - c(\nabla \myalpha \cdot v ,  g)_\Omega - d(\myalpha\,\div v,  g)_\Omega + (f,g)_\Omega, \label{eq:lin2w}
\end{align}
for all test functions $u \in L^2(\Omega)$ and $g \in H_0^1(\Omega)$, and for all $0 \le t \le T$.
By reversing the steps arising in the derivation of this weak form of the system, one can see that, vice versa, any sufficiently regular solution of \eqref{eq:lin1w}--\eqref{eq:lin2w} also satisfies the differential equations \eqref{eq:lin1}--\eqref{eq:lin2}.
The key results of this section are summarized in the following statement. 
\begin{proposition} \label{pro:lin}
Let Assumption~\ref{ass:1} hold, $T>0$, and further assume that  \\[-2em] 
\begin{itemize} \itemindent1em
\item $f \in L^2(0,T;H_0^1(\Omega))$;
\item $v_0 \in H^1_\div(\Omega)$, $h_0 \in H^2(\Omega)$ with $h_0=\div v_0=0$ on $\partial\Omega$;
\item $\myalpha \in L^2(0,T;H^3(\Omega)) \cap L^\infty(0,T;H^2(\Omega))$ with
		$\max\big\{ \|\myalpha\|_{L^\infty(L^\infty)},C_{H^1_0,L^4} \|\nabla \myalpha\|_{L^\infty(L^4)} \big\}< 1/a$, 
\end{itemize} \vspace{-1em}
where $C_{H^1_0,L^4}$ denotes the constant of the Sobolev embedding $id:H_0^1(\Omega) \to L^4(\Omega)$. \\
Then the linearized system \eqref{eq:lin1}--\eqref{eq:lin2} together with \eqref{eq:sys3}--\eqref{eq:sys4} admits a unique solution $(v,h) \in \mathcal{V}_T \times \mathcal{H}_T$ 
and 
\begin{align} \label{eq:linbound}
\|v\|_{\cal V_T}^2 + \|h\|_{\cal H_T}^2 
&\leq C_L \cdot \big(\|v_0\|^2_{H^1_\div} + \|h_0\|_{H^2}^2 + \|f\|_{L^2(H^1)}^2 \big), 
\end{align}
with constant $C_L = \hat C_1 \cdot \exp\big(\hat C_2 T + \hat C_3 \|\nabla \myalpha\|_{L^\infty(L^\infty) \cap L^2(W^{1,4})}\big)$ and $\hat C_i$ depending only on the bounds in the assumptions. 
\end{proposition}

\begin{proof}
We follow the standard procedure~\cite{boyer2012mathematical,evans2010partial,roubivcek2013nonlinear}, and use the Galerkin approximation, a priori estimates, and compactness arguments to establish the existence of a solution to \eqref{eq:lin1w}--\eqref{eq:lin2w} with the required regularity properties. 
Let $w^k$, $k \ge 1$ denote the eigenfunctions of the Laplace operator defined by $-\Delta w^k = \lambda w^k$ in $\Omega$ with $w^k|_{\partial\Omega}=0$. 
Since $\Omega$ is bounded with boundary in $C^{3,1}$, we can conclude that $w^k \in H^{4}(\Omega)$ with $w^k = \Delta w^k=0$ on $\partial\Omega$, see \cite[Theorem 20.4]{Wloka}. 
We also note that $(w^k)_{k \ge 1}$ forms a complete orthogonal system in $L^2(\Omega)$ and $H_0^1(\Omega) \cap H^2(\Omega)$. 
We denote by $W^k=\text{span}\{w^j : j \le k\}$ the finite-dimensional subspace spanned by the first $k$ eigenfunctions. 
Furthermore, let $\Pi^k$ denote the $L^2$-orthogonal projection onto $W^k$, which is also orthogonal on $H_0^1$ and $H^2 \cap H_0^1$, and hence $\|h_0^k\|_{X} \leq \|h_0\|_{X}$ for any choice of the spaces $X \in \{L^2(\Omega),H_0^1(\Omega),H_0^1(\Omega)\cap H^2(\Omega)\}$; see \cite[Lemma 7.5]{robinson2013infinite}.
\medskip 

\noindent 
\textbf{Step 1: Galerkin approximation.} 
We define $h^k : [0,T] \to W^k$ as the solution of  
\begin{align} \label{eq:galerkin}
(h_t^k,g^k)_\Omega + b \, (\nabla h^k,\nabla g^k)_\Omega 
= - (\div v^k, g^k)_\Omega + a \, (\myalpha h_t^k,g^k)_\Omega - c \, (\nabla \myalpha, v^k g^k)_\Omega- d \, (\div v^k,\myalpha g^k)_\Omega + (f,g^k)_\Omega,
\end{align}
which is assumed to hold for all $0 \le t \le T $ and for all $g^k \in W_k$, with initial value $h^k(0)=\Pi^k h_0$, and velocity component given by the explicit solution formula 
$v^k(t)=v_0-\int_0^t \nabla h^k(s) \ds$. 
The latter implies that the relation
\begin{align} \label{eq:galerkin_v} 
v_t^k=-\nabla h^k 
\end{align} 
holds pointwise in $\Omega \times (0,T)$ and $v^k(0)=v_0$. 
It is not difficult to see that \eqref{eq:galerkin}--\eqref{eq:galerkin_v} amounts to a finite-dimensional system of integro-differential equations. The local in time existence of solutions can thus be proven by standard arguments.
From the regularity of the basis functions $w^k$ and the initial values $h_0$, $v_0$, we further deduce that 
$h^k \in C^1([0,T];H^2(\Omega)) \cap C([0,T];H^4(\Omega))$ and $v^k \in C^1([0,T];H^1_\div(\Omega))$ with homogeneous boundary values $h^k=\Delta h^k = \div v^k=0$ on $\partial\Omega$. 

\medskip 

\noindent 
\textbf{Step~2. Energy estimates.}
We now derive a sequence of energy estimates which essentially follow by testing \eqref{eq:galerkin} appropriately and using \eqref{eq:galerkin_v} where required. 
For ease of notation, we write $\|\cdot\|=\|\cdot\|_{L^2}$ in the sequel and introduce another energy function
\begin{align} \label{eq:estar}
E^*(t) = \frac{1}{2} \Big(\|h^k\|^2 + (1+Kb) \|\nabla h^k\|^2 + (b+L) \|\Delta h^k\|^2 + \|v^k\|^2 + \|\div v^k\|^2 + L \|\nabla \div v^k\|^2\Big), 
\end{align}
with some parameters $L,K>0$ that will be chosen below. This energy will be used to bound several terms in our estimates.

\medskip 
\noindent 
\textbf{Step 2a.} 
We test \eqref{eq:galerkin} with $g^k=h^k$, use integration-by-parts for the third term in \eqref{eq:galerkin}, and estimate using standard tools, that is, Cauchy-Schwarz, H\"older, and Young inequalities. This gives 
\begin{align*}
\frac{1}{2} \ddt \|h^k\|^2 + b \|\nabla h^k\|^2 
&= (v^k, \nabla h^k)_\Omega + a (\myalpha h_t^k, h^k)_\Omega - c (\nabla \myalpha \cdot v^k, h^k)_\Omega - d (\myalpha \div v^k, h^k)_\Omega + (f,h^k)_\Omega \\
&\le -\frac{1}{2} \ddt \|v^k\|^2 + \delta \|h_t^k\|^2 + \|f\|^2 + C_1\big(1 + 1/\delta + \|\myalpha\|_{L^\infty}^2 + \|\nabla \myalpha\|_{L^\infty}^2\big) \,  E^*(t).
\end{align*}
For the first term on the right-hand side, we used \eqref{eq:galerkin_v}. 
The parameter $\delta>0$ comes from Young's inequality and will be chosen below. Finally, note that the constant $C_1$ depends solely on the bounds of the parameters in Assumption~\ref{ass:1}.

\medskip 

\noindent\textbf{Step 2b.}
We continue by testing \eqref{eq:galerkin} with $g^k=K h_t^k$. This yields
\begin{align*}
K \|h_t^k\|^2 + \frac{Kb}{2} \ddt\|\nabla h^k\|^2 
&= K (\div v^k, h_t^k)_\Omega  +  K a (\myalpha h_t^k, h_t^k)_\Omega - K c (\nabla \myalpha \cdot v^k, h_t^k)_\Omega - Kd (\myalpha \div v^k, h_t^k)_\Omega + K (f,h_t^k)_\Omega \\
&\le K (a \|\myalpha\|_{L^\infty} + \delta) \|h_t^k\|^2 + C_2 \|f\|^2 + C_2'\big(1 + 1/\delta + \|\myalpha\|_{L^\infty}^2 +  \|\nabla \myalpha\|_{L^\infty}^2\big) \, E^*(t).
\end{align*}
The constants $C_2$ and $C_2'$ here depend on the bounds for the parameters and the choice of $K$ and $\delta$ below.

\medskip 

\noindent 
\textbf{Step 2c.} 
Next, we test \eqref{eq:galerkin} with $g^k=-\Delta h^k$. This yields
\begin{align*}
\frac{1}{2} \ddt \|\nabla h^k\|^2 + b \|\Delta h^k\|^2 
&= (\div v^k, \Delta h^k)_\Omega 
- a (\myalpha h_t^k, \Delta h^k)_\Omega + c (\nabla \myalpha \cdot v^k, \Delta h^k)_\Omega + d (\myalpha \div v^k, \Delta h^k)_\Omega - (f,\Delta h^k)_\Omega \\
&\le -\frac{1}{2} \ddt \|\div v^k\|^2 + \delta \|h_t^k\|^2 + C_3 \|f\|^2 + 
C_3'\big(1 + 1/\delta + \|\myalpha\|_{L^\infty}^2 +  \|\nabla \myalpha\|_{L^\infty}^2\big) \, E^*(t).
\end{align*}
For the first term on the right-hand side, we again used \eqref{eq:galerkin}.
The constants $C_3$, $C_3'$ here again depend solely on the bounds for the model parameters and the choice of the small parameter $\delta$; see below.

\medskip 

\noindent 
\textbf{Step 2d.} 
We next test \eqref{eq:galerkin} with $g^k=-\Delta h_t^k$ and use integration--by--parts in space for most of the terms. This leads to 
\begin{align*}
\|\nabla h_t^k\|^2 + \frac{b}{2} \ddt \|\Delta h^k\|^2 
&= (\nabla \div v^k, \nabla h_t^k)_\Omega - a (\nabla \myalpha h_t^k + \myalpha \nabla h_t^k, \nabla h_t^k)_\Omega + c (\nabla^2 \myalpha \cdot v^k + \nabla \myalpha : \nabla v^k, \nabla h_t^k)_\Omega \\
& \qquad \qquad + d(\nabla \myalpha \div v^k + \myalpha \nabla \div v^k, \nabla h_t^k)_\Omega - (\nabla f,\nabla h_t^k)_\Omega \\
&\le (a \|\myalpha\|_{L^\infty} + \delta) \|\nabla h_t^k\|^2 + C_4\big(\|\nabla f\|^2 + \|v_0\|_{H^1}^2\big) \\
&\qquad \qquad + C_4'\big(1 + 1/\delta + \|\myalpha\|_{L^\infty}^2 + \|\nabla \myalpha\|_{L^\infty}^2 + \|\nabla^2 \myalpha\|_{L^4}^2\big) \, E^*(t).
\end{align*}
In the second step, we used $ \nabla v^k(t) = \nabla v(0) + \int_0^t \nabla \dt v^k(s) ds$ and \eqref{eq:galerkin_v} to obtain the bound
\begin{align*}
\|\nabla v^k\|^2 
&\le 2 \|\nabla v_0\|^2 + 2 T \|\nabla^2 h^k\|^2 \le C_4''\big( \|\nabla v_0\|^2 + 2 T \|\Delta h^k\|^2_{L^2}\big).
\end{align*}
In addition, we used $\|\Delta h^k\| \simeq \|h^k\|_{H^2}$ on $H_0^1(\Omega)\cap H^2(\Omega)$, since $\Omega$ was assumed to have a $C^{3,1}$-regular boundary. 
\medskip 

\noindent 
\textbf{Step 2e.}
As a final step, we test \eqref{eq:galerkin} with $g^k=L \Delta^2 h^k$ and integrate by parts in space as before. This leads to 
\begin{align*}
\frac{L}{2}\ddt \|\Delta h^k\|^2 + Lb \|\nabla\!\Delta h^k\|^2 
&= L (\nabla \div v^k, \nabla\!\Delta h^k)_\Omega - La (\nabla \myalpha h_t^k + L\myalpha \nabla h_t^k, \nabla\!\Delta h^k)_\Omega + Lc (\nabla^2 \myalpha \cdot v^k + \nabla \myalpha : \nabla v^k, \nabla\!\Delta h^k)_\Omega \\
& \qquad \qquad + Ld(\nabla \myalpha \div v^k + \myalpha \nabla \div v^k, \nabla\!\Delta h^k)_\Omega - L (\nabla f,\nabla\!\Delta h^k)_\Omega \\
&\le -\frac{L}{2} \ddt \|\nabla \div v^k\|^2 +  L \delta \|\nabla\!\Delta h^k\|^2 + \frac{L a^2}{\delta}\big( \|\myalpha\|_{L^\infty}^2 \|\nabla h_t^k\|^2 + \|\nabla \myalpha\|_{L^4}^2 \|h_t^k\|_{L^4}^2\big)  \\
& \qquad \qquad + C_5\big(\|\nabla f\|^2 + \|\nabla v_0\|^2\big) 
+ C_5'\big(1 + 1/\delta  + \|\myalpha\|_{L^\infty}^2 + \|\nabla \myalpha\|_{L^\infty}^2 + \|\nabla^2 \myalpha\|_{L^4}^2\big) \, E^*(t).
\end{align*}
For the second inequality, we used very similar arguments as in the previous step.
The third term on the right-hand side can be further estimated by $\frac{C' L a^2}{\delta} \|\myalpha\|_{L^\infty \cap W^{1,4}}^2 \big(\|h_t^k\|^2 + \|\nabla h_t^k\|^2\big)$ using Sobolev embeddings.  

\medskip 

\noindent 
\textbf{Step~3. Uniform bounds.}
By adding up the estimates derived in the previous steps, we obtain
\begin{align*}
\frac{1}{2} \ddt 
&\Big( 
\|h^k\|^2 + \|v^k\|^2               %
+ Kb \|\nabla h^k\|^2              %
+ \|\nabla h^k\|^2 + \|\div v^k\|^2  %
+ b \|\Delta h^k\|^2               %
+ L\|\Delta h^k\|^2                 
  + L\|\nabla \div v^k\|^2         %
\Big) \\
&\qquad \qquad + 
\Big(
b \|\nabla h^k\|^2                 %
+ K \|h_t^k\|^2                    %
+ b \|\Delta h^k\|^2               %
+ \|\nabla h_t^k\|^2               %
+ Lb \|\nabla\!\Delta h^k\|^2       %
\Big) \\
& \qquad \qquad - 
\Big( 
K (a \|\myalpha\|_{L^\infty} + \delta) \|h_t^k\|^2                 %
+ \delta \|h_t^k\|^2               %
+ \delta \|h_t^k\|^2               %
+ (a \|\myalpha\|_{L^\infty} 
   + \delta) \|\nabla h_t^k\|^2    %
+ \tfrac{C' a^2 L}{\delta} 
  \|\myalpha\|_{L^\infty \cap W^{1,4}}^2 \big(\|h_t^k\|^2 + \|\nabla h_t^k\|^2\big)
\Big) \\
&\le 
C_1^* \|f\|_{H^1}^2 + C_2^* \|v_0\|_{H^1}^2 + C_3^* \big(1 + 1/\delta  + \|\myalpha\|_{L^\infty}^2 + \|\nabla \myalpha\|_{L^\infty}^2 + \|\nabla^2 \myalpha\|_{L^4}\big) \, E^*(t).
\end{align*}
We now choose $\delta>0$ sufficiently small, $K$ sufficiently large, and finally $L$ sufficiently small, such that the terms in the third line can be absorbed into the second line of the left-hand side. 
In order to be able to do so, we used $a \|h\|_{L^\infty} < 1$. 
Together with the definition of $E^*(t)$, see \eqref{eq:estar}, we then see that
\begin{align*}
\ddt E^*(t) + c^*\big(\|h_t^k\|_{H^1}^2 + \|\nabla\!\Delta h^k\|^2\big) &\le 
C_1^* \|f\|_{H^1}^2 + C_2^* \|v_0\|_{H^1}^2 + C_3^*\big(1 + 1/\delta + \|\myalpha\|_{L^\infty}^2 + \|\nabla \myalpha\|_{L^\infty}^2 + \|\nabla^2 \myalpha\|_{L^4}^2\big) \, E^*(t). 
\end{align*}
Note that the constants $c^*,C_i^*$ depend solely on the bounds for the parameters, and $C_i^*$ additionally depends on the final time $T$.
We then integrate this inequality in time and use the definition of $E^*(t)$ as well as $\|\Delta h^k\| \simeq \|h^k\|_{H^2}$ and $\|\nabla\!\Delta h^k\| \simeq \|h^k\|_{H^3}$, where we employed the  boundary conditions. 
We further rewrite the resulting estimate in the form $z(t)\leq z_0+\int_0^t g(s) z(s) \,\textup{d}s$
with 
\begin{align}
z(t)&=\|h^k(t)\|_{H^2}^2 + \|v^k(t)\|_{H^0_\div}^2
+  \|h^k_t\|_{L^2_t(L^2)}^2  + \|h^k\|_{L^2_t(H^2)}^2 , \\
z_0 &=\|h_0\|_{H^2}^2 + \|v_0\|_{H^1_\div}^2 + \hat C_1  \|f\|_{L^2(H^1)}^2, \\
g(s) &= \hat C_2 + \hat C_3 \big(\|\nabla\myalpha(s)\|_{L^\infty}^2 + \|\nabla^2 \myalpha\|_{L^4}^2\big).
\end{align}
Note that $\hat C_i$ are uniform constants and $g \in L^1(0,T)$ according to our assumptions on $\myalpha$. 
Applying the Gr\"onwall inequality, see \cite[Lemma II.4.10]{boyer2012mathematical}, we thus conclude that
$z(t) \leq z_0 \exp\!\big(\int_0^t g(s) \, \textup{d}s\big)$. Together with the definition of $v^k$, this already yields 
\begin{align*}
\|h^k(t)\|_{H^2}^2  &+  \|v^k(t)\|_{H^1_\div}^2 + \|h^k_t\|_{L^2(H^1)}^2  + \|h^k\|_{L^2(H^3)}^2  \\ 
&\leq \hat C_1 \, \exp\!\big(\hat C_2 T + \hat C_3 \|\nabla \myalpha\|_{L^2(L^\infty) \cap H^1(L^4)}^2  \big) \cdot  \big(\|h_0\|_{H^2}^2 + \|v_0\|_{H^1_\div}^2 +\|f\|_{L^2(H^1)}^2 \big),
\end{align*}
which holds with constants $\hat C_i$ that are independent of the dimension of $W^k$.

\medskip 

\noindent 
\textbf{Step 4: Limit process, existence, and uniqueness.}
The uniform bounds on the solution $(v^k,h^k)$ allow us to extract a sub-sequence which converges in an appropriate sense to a limit function $(\bar v,\bar h)$ satisfying the same bounds. 
Due to the high regularity of the solutions, one can pass to the limit in all terms of \eqref{eq:galerkin} and show that $(\bar v,\bar h)$ satisfies the weak form of \eqref{eq:lin1}--\eqref{eq:lin2} as well as the boundary and initial conditions \eqref{eq:sys3}--\eqref{eq:sys4}; see \cite{roubivcek2013nonlinear} for details of the arguments involved. 
Uniqueness finally follows from the linearity of the equations and employing the energy estimates once again. 
\end{proof}

\section{Local solvability}
\label{sec:local}

We can now establish the local existence of a unique solution to problem \eqref{eq:sys1}--\eqref{eq:sys4} by means of Banach's fixed-point theorem.
To do so, we introduce the balls $\mathcal{V}_T^R=\{v \in \mathcal{V}_T : \|v\|_{\mathcal{V}_T} \le R\}$ and $\mathcal{H}_T^R=\{h \in \mathcal{H}_T : \|h\|_{\mathcal{H}_T} \le R\}$ of radius $R$ in $\mathcal{V}_T$ and $\mathcal{H}_T$, respectively. 
The parameter $R$ will be adapted several times below as needed.
We will prove the following result.
\begin{proposition} \label{pro:local}
Let Assumption~\ref{ass:1} hold. Then for any $T>0$, there exist $\eps=\eps(T)>0$ and $R=R(T)>0$, such that for all initial values $(v_0,h_0) \in H^1_\div(\Omega) \times H^2(\Omega)$ with $h_0=\div v=0$ on $\partial\Omega$ and such that $\cal E(v_0,h_0) \le \eps$, the system \eqref{eq:sys1}--\eqref{eq:sys4} has a unique local solution $(v,h) \in \cal V_T^R \times \cal H_T^R$. Moreover, $\|(v,h)\|_{\cal V_T \times \cal H_T} \le C(T) \, \cal E(v_0,h_0)$ with constant $C(T)$ depending only on $T$. 
\end{proposition}
\begin{proof}
We consider the mapping 
\begin{align} \label{eq:fpop}
\mathcal{F} : \mathcal{V}_T^R \times  \mathcal{H}_T^R &\to \mathcal{V}_T \times \mathcal{H}_T, \qquad (\tilde v,\tilde h) \mapsto (v,h),
\end{align}
where $(v,h)$ is the solution to \eqref{eq:lin1}--\eqref{eq:lin2} and \eqref{eq:sys3}--\eqref{eq:sys4}, with parameter functions $\myalpha = \tilde h$, and $f=0$. 
We note that $\cal F(\tilde h,\tilde v)$ does not depend on the velocity $\tilde v$. 
Any fixed-point $(\bar v,\bar h) \in \mathcal{V}_T^R \times \mathcal{H}_T^R$ of $\mathcal{F}$ can be considered a solution \eqref{eq:weak1}--\eqref{eq:weak2} and to additionally satisfy the initial and boundary conditions \eqref{eq:sys3}--\eqref{eq:sys4}. 
By the regularity of the solutions to the linearized problem \eqref{eq:lin1}--\eqref{eq:lin2}, one can see that any fixed-point $(\bar v,\bar h)$ in fact also solves the original equations \eqref{eq:sys1}--\eqref{eq:sys2} in strong form. 
It thus remains to verify the assumptions of Banach's fixed-point theorem, which is what we do in the following. 

\medskip 

\noindent 
\textbf{Step~1.}
For any $T>0$ and $R>0$, the set $\mathcal{V}_T^R \times \mathcal{H}_T^R$ is non-empty and closed in the topology of $\mathcal{V}_T \times \mathcal{H}_T$.
Moreover, for $R$ sufficiently small, any function $\myalpha = \tilde h \in \mathcal{H}_T^R$ satisfies the assumptions of Proposition~\ref{pro:lin}, and hence $\mathcal{F}$ is well-defined.

\medskip 
\noindent 
\textbf{Step~2. Self-mapping.}
Let $(\tilde v,\tilde h) \in \mathcal{V}_T^R \times \mathcal{H}_T^R$ with $R$ sufficiently small. In addition, let $(v,h)=\mathcal{F}(\tilde v,\tilde h)$ be the unique solution of the linearized problem \eqref{eq:lin1}--\eqref{eq:lin2} with \eqref{eq:sys3}--\eqref{eq:sys4}, $\myalpha = \tilde h$, and $f=0$. 
From the estimates of Proposition~\ref{pro:lin}, we see that 
\begin{align} %
&\|(v,h)\|_{\cal V_T \times \cal H_T}^2 \leq C_L \cdot \cal E(v_0,h_0).
\end{align}
Using $\|\tilde h\|_{\cal H_T} \le R$ and $\myalpha=\tilde h$, we can bound 
$\|\myalpha\|_{L^2(L^\infty) \cap H^1(L^4)} \le C R$ by using appropriate Sobolev embeddings~\cite{evans2010partial,roubivcek2013nonlinear}, and hence 
$C_L = \hat C_1 \exp(\hat C_2 T + \hat C_3 C R)$. 
We may thus choose $\eps=\eps(R,T)$ sufficiently small, such that $\|(v,h)\|_{\cal V_T \times \cal H_T} \le R$ for all initial values satisfiying $\cal E(v_0,h_0) \le \eps$.
Under this restriction, $\cal F$ is a self-mapping on $\cal V_T^R \times \cal H_T^R$.

\medskip 
\noindent 
\textbf{Step~3. Contraction.}
Let the two tuples $(\tilde v_1, \tilde h_1), (\tilde v_2,  \tilde h_2) \in \cal V_T^R \times \cal H_T^R$ be given, $R>0$ be sufficiently small, and set 
\begin{align*} 
(v_1,h_1):=\cal F(\tilde v_1, \tilde h_1), \qquad  (v_2,h_2):=\cal F(\tilde v_2, \tilde h_2).
\end{align*}
We abbreviate the differences by $(\tilde v,\tilde h):=(\tilde v_1-\tilde v_2,\tilde h_1-\tilde h_2)$ and $(v,h):=(v_1-v_2,h_1-h_2)$, and note that $(v(0),h(0))=(0,0)$, since the functions $(v_i,h_i)$ satisfy the same initial conditions \eqref{eq:sys4}. 
From the definition of $\cal F$ via \eqref{eq:lin1}--\eqref{eq:lin2}, one can further see that 
\begin{align} \label{Eq:LinearizedDiff}            (v_t,u)_\Omega +(\nabla  h, u)_\Omega &=0, \\
((1-a\tilde h_1)h_t,g)_\Omega + b (\nabla  h,\nabla g)_\Omega-(v,\nabla g)_\Omega  &= a(\tilde h h_{2,t},g)_\Omega- c(\nabla \tilde h_1 \cdot v  ,g)_\Omega -c(\nabla  \tilde h \cdot v_2   ,g)_\Omega \\ &\quad - d(  \tilde h_1\, \div v, g)_\Omega - d(\tilde h\,  \div v_2, g)_\Omega, \notag
\end{align}
for any $u \in L^2(\Omega)^n$, $g \in H^1(\Omega)$, and all $0 \le t \le T$.
This amounts to the weak form of the linearized system \eqref{eq:lin1}--\eqref{eq:lin2} with 
\begin{alignat}{2}
 \myalpha=\tilde h_1 
 \qquad \text{and} \qquad  
  f =a \tilde h h_{2,t} - d\tilde h \div v_2 - c \nabla \tilde h \cdot v_2.
\end{alignat}
In order to apply Proposition~\ref{pro:lin}, we estimate $f$ in the norm of $L^2(H^1)$. 
We treat the three terms independently. First, note that 
\begin{align*}
\|\tilde h h_{2,t}\|_{L^2(H^1)} 
\le \|\tilde h\|_{L^\infty(L^\infty)} \|h_{2,t}\|_{L^2(H^1)} + \|\nabla \tilde h\|_{L^\infty(L^4)} \|h_{2,t}\|_{L^2(L^4)} \le C(T) R \|\tilde h\|_{\cal H_T},
\end{align*}
where we applied Hölder and Sobolev inequalities, and the uniform bounds for $\tilde h_i \in \cal H_T^R$. 
For the second term, we obtain
\begin{align*}
\|\tilde h \div v_2\|_{L^2(H^1)} 
\le \|\tilde h\|_{L^\infty(L^\infty)} \|\div v_2\|_{L^2(H^1)} + \|\nabla \tilde h\|_{L^\infty(L^4)} \|\div v_2\|_{L^2(L^4)} \le C(T) R \|\tilde h\|_{\cal H_T},
\end{align*}
using similar arguments as before. The third term can finally be bounded by
\begin{align*}
\|\nabla \tilde h \cdot v_2\|_{L^2(H^1)} 
\le \|\nabla \tilde h\|_{L^2(L^\infty)} \|v_2\|_{L^\infty(H^1)} + \|\nabla^2 \tilde h\|_{L^\infty(L^2)} \|v_2\|_{L^2(L^\infty)} \le C(T) R \|\tilde h\|_{\cal H_T}. 
\end{align*}
By applying Proposition~\ref{pro:lin}, we thus immediately obtain
\begin{align}
\|\cal F(\tilde v_1,\tilde h_1) - \cal F(\tilde v_2,\tilde h_2)\|_{\cal V_T \times \cal H_T} 
= \|(v,h)\|_{\cal V_T \times \cal H_T}
\le C_L C(T) R \|(\tilde v_1-\tilde v_2, \tilde h_1-\tilde h_2)\|_{\cal V_T \times \cal H_T}.
\end{align}
By choosing $R$ sufficiently small, we get a contraction, and we define $\eps(T) = \eps(T,R)$ to retain the self-mapping property. 
In summary, we have shown that $\cal F$ is a contractive self-mapping on $\cal V_T^R \times \cal H_T^R$ for this choice of $R$, whenever $\cal E(v_0,h_0) \le \eps(T)$. 
This yields the existence of a unique fixed-point $(v,h)$ of $\cal F$ in $\cal V_T^R \times \cal H_T^R$, and hence a local solution to \eqref{eq:sys1}--\eqref{eq:sys4}. \end{proof}

\begin{remark}
With the same arguments, one can show that the unique local solution $(v,h)$ of problem \eqref{eq:sys1}--\eqref{eq:sys4} also depends continuously on the (initial) data. Hence, the problem is (locally) well-posed. The details are left to the reader.    
\end{remark}

\section{Global solutions and decay estimate} 
\label{sec:global}

As a next step of our analysis, we now derive a-priori estimates for the solution of \eqref{eq:sys1}--\eqref{eq:sys3} that are independent of the time horizon $T$. As a consequence, the local solution can be extended to infinity. Moreover, the estimates will allow us to establish exponential decay to equilibrium. 
Let us start with stating and proving the time-independent a-priori bounds. 
\begin{proposition}\label{pro:global} 
Let Assumption~\ref{ass:1} hold. Then there exist constants $\eps_\infty$, $C_\infty>0$, such that for all $T>0$ and initial values satisfying $\cal E(v_0,h_0) \le \eps_\infty$, the local solutions $(v,h) \in \cal V_T \times \cal H_T$ provided by Proposition~\ref{pro:local} satisfy
\begin{align} \label{eq:global}
&\|h(t)\|_{H^2}^2 + \|\div v(t)\|_{H^1}^2+ \int_0^t \|h\|_{H^3}^2 + \|h_t\|_{H^1}^2 + \|\div v\|_{H^1}^2 \, ds  \leq C_\infty \big(\|\div v_0\|_{H^1}^2+\|h_0\|_{H^2}^2 \big) \quad \forall t \in (0,T]. 
\end{align}
Moreover, the energy at a time $t$ can be estimated by 
$\cal E(v(t),h(t))
\leq C_0\cal E(v_0,h_0)$ for all $t \in (0,T]$.
\end{proposition}
\begin{proof}
Fix $T>0$ and let $(v,h) \in \cal V_T \times \cal H_T$ denote a local solution of \eqref{eq:sys1}--\eqref{eq:sys3} on the time interval $[0,T]$. 
From the a-priori estimates of Proposition~\ref{pro:local}, we know that $\|(v,h)\|_{\cal V_T \times \cal H_T}^2:=\|v\|_{\cal V_T}^2 + \|h\|_{\cal H_T}^2 \leq R^2$, with $R>0$ depending monotonically on the norm of the initial conditions. 
In particular, we may assume $R=R(\eps,T)$ sufficiently small by conditions on the initial values.
From the regularity of the solution, we deduce validity of 
\begin{align}
v_t &= -\nabla h &&\text{ in } L^2(0,T;H^2(\Omega)^n), \label{Eq:Global1} \\
(1-ah)h_t - b\Delta h &= - \div v - c v \cdot \nabla h - d\,h\,\div v &&\text{ in } L^2(0,T;H^1(\Omega)). \label{Eq:Global2}
\intertext{By formal differentiation of these equations, one can further deduce validity of the additional equations}
\Delta v_t &= -\Delta\!\nabla h &&\text{ in } L^2(0,T;L^2(\Omega)^n), \label{Eq:TestGlobal3}
\\
\nabla \div v_t &= -\nabla\!\Delta h &&\text{ in } L^2(0,T;L^2(\Omega)^n), \label{Eq:TestGlobal2}
\\
(1-ah)\nabla h_t - a \nabla h h_t - b\nabla\!\Delta h &= - \nabla \div v - c\nabla( v \cdot \nabla h) -d \nabla( h\div v) &&\text{ in } L^2(0,T;L^2(\Omega)^n). \label{Eq:TestGlobal1}
\end{align} 
By formal testing of these equations, similar to the proof of Proposition~\ref{pro:local}, we will now 
establish $T$-independent a-priori bounds.

\medskip 

\noindent
\textbf{Step 1.} 
We multiply \eqref{Eq:Global1} with $-\nabla \div v$ and \eqref{Eq:Global2} with $-\Delta h$, integrate over $\Omega_t = \Omega \times (0,t)$, and sum up to obtain
\begin{align*}
\frac{1}{2} \Big(\|\nabla h(t)\|_{L^2}^2 &+  \|\div v(t)\|_{L^2}^2\Big) + b \|\Delta h\|_{L^2_t(L^2)}^2 \\
&= \frac{1}{2} \Big( \|\nabla h_0\|_{L^2}^2 + \|\div v_0\|_{L^2}^2\Big)  
- a (h h_t,\Delta h)_{\Omega_t}   + c(\nabla h, v \Delta h)_{\Omega_t} + d(\div v,h \Delta h)_{\Omega_t} \\
&\leq \frac{1}{2} \Big(\|\nabla h_0\|_{L^2}^2 +\|\div v_0\|_{L^2}^2\Big) +a\|h\|_{L^2_t(L^\infty)} \|h_t\|_{L^\infty(L^2)} \|\Delta h\|_{L^2_t(L^2)} \\
&\qquad \qquad + c \|\nabla h\|_{L^2_t(L^4)} \|v\|_{L^\infty(L^4)} \|\Delta h\|_{L^2_t(L^2)} 
+d \|\div v\|_{L^\infty(L^2)} \|h\|_{L^2_t(L^\infty)} \|\Delta h\|_{L^2_t(L^2)}. 
\end{align*}
Here and below, we use $(u,u')_{\Omega_t} = \int_0^t (u,u')_\Omega \, ds$ and write $\|u\|_{L^p_t(X)} = \|u\|_{L^p(0,t;X)}$ and $\|u\|_{L^p(X)}=\|u\|_{L^p(0,T;X)}$ for abbreviation. 
By Sobolev embeddings~\cite{evans2010partial,roubivcek2013nonlinear}, the uniform bounds for the solution $(v,h)$ on $[0,T]$, and $\|h\|_{H^2} \le C' \|\Delta h\|_{L^2}$, we get
\begin{align*}
\frac{1}{2} \Big( \|\nabla h(t)\|_{L^2}^2   + \|\div v(t)\|_{L^2}^2 \Big)
+ b \|\Delta h\|_{L^2_t(L^2)}^2
\leq \frac{1}{2} \Big(\|\nabla h_0\|_{L^2}^2 +\|\div v_0\|_{L^2}^2\Big) +  C R \|\Delta h\|_{L^2_t(L^2)}^2,
\end{align*}
with a constant $C$ depending only on the domain $\Omega$ and the bounds for the model parameters. 
The last term in this estimate can be absorbed into the right-hand side, provided that $R=R(T,\eps)$ is sufficiently small, which can be done by a smallness condition on the initial values. 
In particular, we may assume $C R \le b/2$ and multiply the inequalities by $2/b$ to obtain the following 
\begin{align} \label{Eq:GlobalEst:1} 
\frac{1}{b} \Big(\|\nabla h(t)\|_{L^2}^2 + \|\div v(t)\|_{L^2}^2\Big) + 2 \|\Delta h\|_{L^2_t(L^2)}^2  \leq  \frac{1}{b} \Big(\|\nabla h_0\|_{L^2}^2 + \|\div v_0\|_{L^2}^2\Big). 
\end{align}

\noindent
\textbf{Step 2.} 
We test \eqref{Eq:TestGlobal2} with $\nabla \div v$ and \eqref{Eq:TestGlobal1} with $-\nabla\!\Delta h \in L^2(0,T;L^2(\Omega))$. 
In a similar manner to before, we obtain 
\begin{align*}
\frac{1}{2} \Big( \|\Delta h(t)\|_{L^2}^2 &+ \|\nabla \div v(t)\|_{L^2}^2\Big) 
+ b\|\nabla\!\Delta h\|_{L^2_t(L^2)}^2 
\\ 
&=\frac{1}{2} \Big(\|\Delta h_0\|_{L^2}^2+\|\nabla \div v_0\|_{L^2}^2 \Big)  
- a( h \nabla h_t,\nabla\!\Delta h)_{\Omega_t} - a (\nabla h h_t,\nabla\!\Delta h)_{\Omega_t}  \\ &\qquad \qquad   + c (\nabla v \cdot \nabla h,\nabla\!\Delta h)_{\Omega_t}+ c(v \nabla^2 h,\nabla\!\Delta h)_{\Omega_t}
+ d(\nabla h \div v,\nabla\!\Delta h)_{\Omega_t} + d(h\nabla \div v,\nabla\!\Delta h)_{\Omega_t} \\
&\leq \frac{1}{2} \Big( \|\Delta h_0\|_{L^2}^2+ \|\nabla \div v_0\|_{L^2}^2 \Big) + a \|h\|_{L^\infty_t(L^\infty)} \|\nabla h_t\|_{L^2_t(L^2)} \|\nabla\!\Delta h\|_{L^2_t(L^2)}  + a \|\nabla h\|_{L^2_t(L^\infty)} \|h_t\|_{L^\infty_t(L^2)} \|\nabla\!\Delta h\|_{L^2_t(L^2)}  \\  
&\qquad \qquad   +c \|\nabla v\|_{L^\infty_t(L^2)} \|\nabla h\|_{L^2_t(L^\infty)} \|\nabla\!\Delta h\|_{L^2_t(L^2)}+ c \|v\|_{L^\infty_t(L^4)} \|\nabla^2 h\|_{L^2_t(L^4)} \|\nabla\!\Delta h\|_{L^2_t(L^2)}  \\  
&\qquad \qquad +d \|\nabla h\|_{L^2_t(L^\infty)} \|\div v\|_{L^\infty_t(L^2)} \|\nabla\!\Delta h\|_{L^2_t(L^2)} + d\|h\|_{L^2_t(L^\infty)} \|\nabla \div v\|_{L^\infty_t(L^2)} \|\nabla\!\Delta h\|_{L^2_t(L^2)} \\
&\leq \frac{1}{2} \Big( \|\Delta h_0\|_{L^2}^2+ \|\nabla \div v_0\|_{L^2}^2 \Big) + CR \|\nabla\!\Delta h\|_{L^2_t(L^2)}^2 + C'R \|\nabla h_t\|_{L^2_t(L^2)}^2,
\end{align*}
with constants $C$, $C'$ again depending only on the domain $\Omega$ and the bounds for the model parameters. 
Requireing $R$ sufficiently small, the third term on the right hand side can be absorbed into the left hand side leading to 
\begin{align} 
\label{Eq:GlobalEst:2}
\|\Delta h(t)\|_{L^2}^2 + \|\nabla \div v(t)\|_{L^2}^2  + b \|\nabla\!\Delta h\|_{L^2_t(L^2)}^2 \leq \|\Delta h_0\|_{L^2}^2 + \|\div v_0\|_{H^1}^2  + 2C'R \|\nabla h_t\|_{L^2_t(L^2)}^2.
\end{align}
The last term can be compensated by the next estimate.

\medskip
\noindent\textbf{Step 3.}  
We now multiply \eqref{Eq:TestGlobal1} with $\nabla h_t$ and integrate over $\Omega_t$.
This yields
\begin{align*} %
((1-ah)\nabla h_t,\nabla h_t)_{\Omega_t} &- a( \nabla h h_t,\nabla h_t)_{\Omega_t} - b(\nabla\!\Delta h,\nabla h_t)_{\Omega_t}  \\ &= -(\nabla \div v,\nabla h_t)_{\Omega_t} - c(\nabla( v \cdot \nabla h),\nabla h_t)_{\Omega_t} -d (\nabla( h\div v),\nabla h_t)_{\Omega_t}.
\end{align*}
We can estimate all nonlinear terms using Hölder inequalities and use the a-priori bounds for $(v,h)$. This leads to 
\begin{align*}
\big(1&-a\|h\|_{L^\infty(H^2)}\big) \|\nabla h_t\|_{L^2_t(L^2)}^2 + \frac{b}{2} \|\Delta h(t)\|_{L^2}^2 \\
&\leq  \frac{b}{2} \|\Delta h_0\|^2_{L^2} -(\nabla \div v,\nabla h_t)_{\Omega_t} + \|\nabla h_t\|_{L^2_t(L^2)} \cdot \Big(  a \|\nabla h\|_{L^\infty_t(L^4)} \|h_t\|_{L^2_t(L^4)}  +  c\| \nabla  v\|_{L^\infty_t(L^2)} \|\nabla h\|_{L^2_t(L^\infty)}    \\
&\qquad \qquad \qquad \qquad   
+ c \|v\|_{L^\infty_t(L^4)} \|\nabla^2 h\|_{L^2_t(L^4)} + d\|\nabla h\|_{L^2_t(L^4)} \|\div v\|_{L^\infty_t(L^4)}+d \|h\|_{L^2_t(L^\infty)} \|\nabla \div v\|_{L^\infty_t(L^2)} \Big)   \\
&\le \frac{b}{2} \|\Delta h_0\|^2_{L^2} -(\nabla \div v,\nabla h_t)_{\Omega_t} 
+ \Big(\frac{1}{8} +CR\Big) \|\nabla h_t\|^2_{L^2_t(L^2)} + C' R^2 \|\nabla\!\Delta h\|^2_{L^2_t(L^2)}.
\end{align*}
The constants $C$, $C'$ again only depend on $\Omega$ and the model parameters.
For the second term on the right-hand side, we integrate by parts in space and time and use \eqref{Eq:Global1} and $\div v=0$ on $\partial\Omega$ to obtain 
\begin{align*} 
-(\nabla \div v,\nabla h_t)_{\Omega_t} &=(\div v,\Delta h_t)_{\Omega_t}  
= -(\div v_t, \Delta h)_{\Omega_t}
+ (\div v(t), \Delta h(t))_\Omega - (\div v_0, \Delta h_0)_\Omega \\
&\leq \|\Delta h\|^2_{L^2_t(L^2)}+\frac{1}{2} \Big(\|\div v(t)\|_{L^2}^2 + \|\Delta h(t)\|_{L^2}^2\Big) + \frac{1}{2}\Big(\|\div v_0\|_{L^2}^2 + \|\Delta h_0\|_{L^2}^2\Big). 
\end{align*}
Combination with the previous estimate, noting that $a \|h\|_{L^\infty(L^2)} \le a C R$, and choosing $R$ sufficiently small, we obtain 
\begin{align} \label{Eq:GlobalEst:3}
\frac{b}{2} \|\Delta h(t)\|^2_{L^2} 
+ \frac{3}{4} \|\nabla h_t\|_{L^2_t(L^2)} 
\le \frac{1}{2} \Big( (1+b) \|\Delta h_0\|^2_{L^2} + \|\div v_0\|^2_{L^2} \Big)  
+ \|\Delta h\|^2_{L^2_t(L^2)} 
+ \frac{1}{2} \Big( \|\div v(t)\|^2_{L^2} + \|\Delta h(t)\|^2_{L^2} \Big). 
\end{align}
The last three terms can be absorbed by terms in the previous estimates.

\medskip 
\noindent 
\textbf{Step~4.}
By adding up the estimates \eqref{Eq:GlobalEst:1}--\eqref{Eq:GlobalEst:3}, we obtain 
\begin{align*}
\frac{1}{b} \Big(\|\nabla h(t)\|^2_{L^2} &+  \|\div v(t)\|^2_{L^2} \Big)+ 2 \|\Delta h\|^2_{L^2(L^2)} 
+ \|\Delta h\|^2_{L^2} + b\|\nabla\!\Delta h\|^2_{L^2_t(L^2)} 
+ \|\nabla \div v\|^2_{L^2} 
+ \frac{b}{2} \|\Delta h\|^2_{L^2} 
+ \frac{3}{4} \|\nabla h_t\|^2_{L^2}  \\
&\le 
\frac{1}{b} \Big( \|\nabla h_0\|^2_{L^2} + \|\div v_0\|^2_{L^2}  \Big)
+ \|\Delta h_0\|^2_{L^2} + \|\div v_0\|_{L^2}^2 + 2 C' R \|\nabla h_t\|^2_{L^2(L^2)}  \\
& \qquad \qquad 
+ \Big(\frac{1}{2} + \frac{b}{2}\Big) \|\Delta h_0\|^2_{L^2} + \frac{1}{2} \|\div v_0\|^2 + \|\Delta h\|^2_{L^2(L^2)} 
+ \frac{1}{2} \|\div v\|^2_{L^2} + \frac{1}{2} \|\Delta h\|^2_{L^2}.
\end{align*}
Assuming $2 C' R \le 1/4$, we see that all terms on the right hand side not involving initial values can be absorbed into the left hand side of the estimate, and we arrive at 
\begin{align} \notag
\|\nabla h(t)\|^2_{L^2}  &+ \|\Delta h\|^2_{L^2} 
+ \|\div v(t)\|^2_{L^2} + \|\nabla \div v(t)\|^2_{L^2} 
\|\Delta h\|^2_{L^2_t(L^2)} 
+ \|\nabla\!\Delta h\|^2_{L^2_t(L^2)} 
+ \|\nabla h_t\|^2_{L^2_t(L^2)} \\
&\le C_\infty' \Big(
\|\nabla h_0\|^2_{L^2}
+ \|\Delta h_0\|^2_{L^2} 
+ \|\div v_0\|^2_{L^2}
+ \|\nabla \div v_0\|^2_{L^2} 
\Big). 
\label{eq:final}
\end{align}
Note that the constant $C_\infty'$ here only depends on the domain $\Omega$ and the model parameters, and the estimate is valid for all solutions $(v,h)$ of \eqref{eq:sys1}--\eqref{eq:sys4} on $[0,T]$ with $T>0$ fixed and initial values satisfying $\cal E(v_0,h_0) \le \eps_\infty$ sufficiently small.

\medskip 
\noindent 
\textbf{Step~5.}
A quick inspection of the arguments used in Steps~1--4 shows that we have actually proven the following statement: There exists $R>0$, $\eps(R)>0$, and $T(R)>0$,  such that for all $0<T \le T(R)$ and initial values $(v_0,h_0)$ satisfying $\cal E(v_0,h_0) \le \eps(R)$, all local solutions of \eqref{eq:sys1}--\eqref{eq:sys4} on $[0,T]$ satisfy the above estimate. 
Without loss of generality, we may assume $C_\infty' \eps(R) \le R/2$; otherwise we simply reduce $\eps(R)$ accordingly. 
Furthermore, we can choose $T(R) = \sup\{T<0: \text{such that the estimate \eqref{eq:final} holds}\}$.
Now assume $T=T(R)<\infty$ and let $(v,h)$ be an arbitrary solution  on $[0,T]$ under consideration. 
Then by \eqref{eq:final} and the assumption on $\eps(R)$, we see that $(v,h) \in \cal V_T^{R/2} \times \cal H_T^{R/2}$, and we may extend the solution in time, say to an interval $[0,T+\delta]$. 
By the arguments used for proving local solvability, we see that $(v,h) \in \cal V_T^R \times \cal H_T^R$ for $\delta>0$ sufficiently small. Moreover, the above estimate \eqref{eq:final}  still holds. 
This is a contradiction to assuming $T(R)$ finite; hence $T(R)=\infty$, i.e., the local solutions provided by Proposition~\ref{pro:local} can be extended to the whole time interval and the estimate \eqref{eq:final} still remains true.
The proof is concluded by utilizing that the norms in \eqref{eq:final} and \eqref{eq:global} are equivalent.
\end{proof}

\subsubsection*{Proof of Theorem~\ref{thm:main}}
Proposition~\ref{pro:global} establishes the existence of a global solution $(v,h)$ to \eqref{eq:sys1}--\eqref{eq:sys4} on $[0,\infty)$ together with uniform bounds \eqref{eq:global}. 
Global uniqueness follows already from the local uniqueness result. 
To prove the decay estimate, we proceed as follows:
With the same arguments as used for proving \eqref{eq:global}, one can show that 
\begin{align*}
\|h(t)\|_{H^2}^2 + \|\div v(t)\|_{H^1}^2 + \int_s^t \|h(r)\|^2_{H^3} + \|h_t(r)\|^2_{H^1} + \|\div v(r)\|^2_{H^2} \, dr \le C_\infty \Big(\|h(s)\|_{H^2}^2 + \|\div v(s)\|^2_{H^2}\Big) 
\end{align*}
for all $0 \le s \le t < \infty$. By dropping some of the terms on the left hand side and taking the limit $t \to \infty$, we thus obtain 
\begin{align*}
\int_s^\infty E(r) \, dr \le C_\infty E(s), \qquad \text{for all } s \ge 0,
\end{align*}
where we used $E(s) = \|h(s)\|^2_{H^2} + \|\div v(s)\|^2_{H^1}$ for abbreviation. 
The decay estimate of Theorem~\ref{thm:main} then follows immediately by application of the Haraux--Lagnese inequality; see e.g.~\cite[Theorem 1.5.9]{qin2017analytic}. \qed

\section{Numerical illustration} \label{sec:num}

In this section, we discuss the structure-preserving discretization of the nonlinear acoustic wave propagation models under investigation. After this, we also present some simulation results for the Westervelt, the Kutsnezov, and the Rasmussen equation, illustrating their similarities and differences.

\subsection{Structure-preserving discretization}

For the space discretization, one can use a conforming Galerkin approximation by mixed finite elements to approximate the functions $(v,h)$. In view of our analysis, a natural choice of spaces is
\begin{align} \label{eq:spaces}
\mathbb{V}_h = P_k(\cal T_h)^n \qquad \text{and} \qquad 
\mathbb{H}_h = P_{k+1}(\cal T_h) \cap H_0^1(\Omega),
\end{align}
where $P_k(\cal T_h)$ is the space of piecewise polynomials of order $\le k$ defined over some conforming simplicial mesh $\cal T_h$ of $\Omega$. 
The semi-discretization of \eqref{eq:sys1}--\eqref{eq:sys4} then reads as follows: Find $(v_h,h_h) : [0,T] \to \mathbb{V}_h \times \mathbb{H}_h$ such that $v_h(0)=v_{0,h}$, $h_h(0)=h_{h,0}$ given, and such that
\begin{align}
(v_{h,t},w_h)_\Omega + (\nabla h_h,w_h)_\Omega &= 0 \label{eq:weak1h}\\
((1-a h_h) h_{h,t},q_h)_\Omega - (v_h,\nabla q_h)_\Omega &= -\, b (\nabla h_h, \nabla q_h)_\Omega - (c-d) (v_h \cdot \nabla h_h, q_h)_\Omega + d (h_h, v_h \cdot \nabla q_h)_\Omega \label{eq:weak2h}
\end{align}
for all $w_h \in \mathbb{V}_h$ and $q_h \in \mathbb{H}_h$, and all $0 \le t \le T$. 
Assuming $\|h_{h,0}\|_{L^\infty}$ sufficiently small, the local existence of a unique discrete solution can be deduced from the Cauchy--Lipschitz theorem.
As a direct consequence of the particular structure of the variational identities that define the discrete solution, we obtain the following result. 
\begin{proposition}[Structure-preservation, discrete power balance] $ $\\
Let $\mathbb{V}_h$, $\mathbb{H}_h$ be chosen as above, and let $(v_h,h_h)$ denote a discrete solution of \eqref{eq:weak1h}--\eqref{eq:weak2h}. 
Then 
\begin{align*}
\frac{d}{dt} \cal H(v_h,v_h) 
= - b \|\nabla h_h\|_{L^2}^2 + (2d-c) (v_h \cdot \nabla h_h, h_h)_\Omega.
\end{align*}
In particular, for $c=2d$ the discrete energy is monotonically decreasing over time. 
\end{proposition}
\begin{proof}
Similarly to the continuous level, the result follows by formally differentiating the energy, which yields
\begin{align*}
\frac{d}{dt} \cal H(v_h,h_h) 
&= (v_{h,t}, v_h)_\Omega + ((1-a h_h) h_{h,t})_\Omega.
\end{align*}
The right-hand side amounts to the first terms in \eqref{eq:weak1h}--\eqref{eq:weak2h} with test function $w_h=v_h$ and $q_h=h_h$. 
The result then follows immediately from the two discrete variational identities. 
\end{proof}
\begin{remark}
The power balance and its proof for the continuous problem are inherited automatically by conforming Galerkin discretization in space. 
Noting that the energy $\cal H(v,h)$ is convex whenever $\|h\|_{L^\infty}$ is sufficiently small, one can employ variational time-discretization strategies to obtain corresponding results as well for fully-discrete schemes; see \cite{egger2019structure,van2014port} for details.  
\end{remark}

\subsection{Fully discrete scheme}
To facilitate the reproduction of our numerical tests, we consider a traditional time-discretization scheme in our numerical tests, namely the implicit mid-point rule. 
Let $\Delta t = T/N$ be a fixed time step, and let $t^n = n \Delta t$ be discrete time points. Given $a^n$, $n \ge 0$, we denote by $a_{\Delta t}^{n-1/2} = \frac{1}{\Delta t} (a^{n} - a^{n-1})$ the central difference quotient at the time $t^{n-1/2}$, and write $a^{n-1/2}=\frac{1}{2}(a^{n} + a^{n-1})$ for the average resp. linear interpolation at $t^{n-1/2}$. 
As a fully discrete approximation for \eqref{eq:sys1}--\eqref{eq:sys4}, we then consider the following scheme: Let $v_h^0=v_{h,0}$ and $h_h^0 = h_{h,0}$ given. Find $(v_h^n,h_v^n) \in \mathbb{V}_h \times \mathbb{H}_h$, $n \ge 1$ such that 
\begin{align*}
(v_{h,\Delta t}^{n-1/2},w_h)_\Omega + (\nabla h_h^{n-1/2},w_h)_\Omega &= 0 
\\
((1-a h_h^{n-1/2}) h_{h,\Delta t}^{n-1/2},q_h)_\Omega - (v_h^{n-1/2},\nabla q_h)_\Omega &= -\, b (\nabla h_h^{n-1/2}, \nabla q_h)_\Omega - (c-d) (v_h^{n-1/2} \cdot \nabla h_h^{n-1/2}, q_h)_\Omega + d (h_h^{n-1/2}, v_h^{n-1/2} \cdot \nabla q_h)_\Omega 
\end{align*}
for all $w_h \in \mathbb{V}_h$ and $q_h \in \mathbb{H}_h$, and all $ 0 < n \le N$. 
Let us note that we cannot prove strict energy--dissipation for this time-discretization strategy. Considering the properties of the semi-discrete problem, we, however, still expect a similar (almost) passive behavior of this fully discrete scheme, at least for sufficiently small time steps.

\medskip 

\noindent \textbf{Simulation setup.}
In our numerical tests, we employ the above scheme with polynomial order $k=1$ in \eqref{eq:spaces} for the space discretization. 
The nonlinear systems arising at every time step of the fully discrete scheme are solved by the Newton method.

\subsection{A one-dimensional example}

We study a similar setting as in \cite[Section 7.1.4]   {fritz2018well}, which was originally adapted from \cite[Chapter 5]{kaltenbacher2007numerical}. As a computational domain, we consider the interval $\Omega=(0,0.4)$ and the time span $[0,10^{-4}]$. We use uniform grids in space and time with $\Delta t=10^{-7}$ and $\Delta x=10^{-3}$. 
The model parameters for our computations are chosen as 
$b=6 \cdot 10^{-9}$, $c_0=1500$, $\gamma=6$. 
We recall that $\gamma$ and $c_0^2$ set the parameter $a$ in the equation; see below \cref{eq:model}.
Furthermore, we choose the initial values $$
v_0 = 0 \qquad \text{and} \qquad h_0=3\cdot 10^5 \cdot \exp\!\bigg(\!\!-\frac{(x-0.2)^2}{2\cdot 0.01^2}\bigg),
$$
that is, initially the wave is centered in the space domain and has its maximum at $x=0.2$ with $h(0,0.2)=3 \cdot 10^5$. 

\begin{figure}[htb!]
\centering
    \begin{tikzpicture}[scale=0.8]
	\definecolor{color0}{HTML}{4E79A7}
	\definecolor{color1}{HTML}{F28E2B}
	\definecolor{color2}{HTML}{E15759}
	\definecolor{color3}{HTML}{76B7B2}
	\definecolor{color4}{HTML}{59A14F}
	\definecolor{color5}{HTML}{EDC948}
	\definecolor{color6}{HTML}{B07AA1}
	\definecolor{color7}{HTML}{FF9DA7}
	\definecolor{color8}{HTML}{9C755F}
	\definecolor{color9}{HTML}{BAB0AC}
	\begin{axis}[
		height=.3\textheight,
        width=.46\textwidth,
		xmin=0.1, xmax=0.3,
		ymin=0, ymax=1.5e5,
		ytick={0,0.5e5,1e5,1.5e5},
		xtick={0.1,0.2,0.3},
        xticklabels={\phantom{0.1},\phantom{0.2},\phantom{0.3}},
		]
		\addplot[ultra thick,color2,smooth,each nth point=1] table [x="Points:0", y="f_14-0", col sep=comma] {Data/wes_80.txt};
		\addplot[ultra thick,color1,smooth,each nth point=1] table [x="Points:0", y="f_14-0", col sep=comma] {Data/kuz_80.txt};
  		\addplot[ultra thick,color0,smooth,each nth point=1] table [x="Points:0", y="f_14-0", col sep=comma] {Data/ras_80.txt};
	\end{axis}
\end{tikzpicture}
\begin{tikzpicture}[scale=0.8]
	\definecolor{color0}{HTML}{4E79A7}
	\definecolor{color1}{HTML}{F28E2B}
	\definecolor{color2}{HTML}{E15759}
	\definecolor{color3}{HTML}{76B7B2}
	\definecolor{color4}{HTML}{59A14F}
	\definecolor{color5}{HTML}{EDC948}
	\definecolor{color6}{HTML}{B07AA1}
	\definecolor{color7}{HTML}{FF9DA7}
	\definecolor{color8}{HTML}{9C755F}
	\definecolor{color9}{HTML}{BAB0AC}
	\begin{axis}[
		height=.3\textheight,
        width=.46\textwidth,
		legend pos=outer north east,
		legend cell align={left},
		xmin=0.1, xmax=0.3,
		ymin=0, ymax=1.5e5,
		ytick={0,0.5e5,1e5,1.5e5},
		xtick={0.1,0.2,0.3},
          scaled y ticks=false,
          xticklabels={\phantom{0.1},\phantom{0.2},\phantom{0.3}},
          yticklabels={},
		]
		\addplot[ultra thick,color2,smooth,each nth point=1] table [x="Points:0", y="f_14-0", col sep=comma] {Data/wes_120.txt};
		\addlegendentry{Westervelt}
		\addplot[ultra thick,color1,smooth,each nth point=1] table [x="Points:0", y="f_14-0", col sep=comma] {Data/kuz_120.txt};
		\addlegendentry{Kuznetsov}
  		\addplot[ultra thick,color0,smooth,each nth point=1] table [x="Points:0", y="f_14-0", col sep=comma] {Data/ras_120.txt};
		\addlegendentry{Rasmussen}
	\end{axis}
\end{tikzpicture} \\
\begin{tikzpicture}[scale=0.8]
	\definecolor{color0}{HTML}{4E79A7}
	\definecolor{color1}{HTML}{F28E2B}
	\definecolor{color2}{HTML}{E15759}
	\definecolor{color3}{HTML}{76B7B2}
	\definecolor{color4}{HTML}{59A14F}
	\definecolor{color5}{HTML}{EDC948}
	\definecolor{color6}{HTML}{B07AA1}
	\definecolor{color7}{HTML}{FF9DA7}
	\definecolor{color8}{HTML}{9C755F}
	\definecolor{color9}{HTML}{BAB0AC}
	\begin{axis}[
		height=.3\textheight,
        width=.46\textwidth,
		xmin=0.1, xmax=0.3,
		ymin=0, ymax=1.5e5,
		ytick={0,0.5e5,1e5,1.5e5},
		xtick={0.1,0.2,0.3},
		]
		\addplot[ultra thick,color2,smooth,each nth point=1] table [x="Points:0", y="f_14-0", col sep=comma] {Data/wes_160.txt};
		\addplot[ultra thick,color1,smooth,each nth point=1] table [x="Points:0", y="f_14-0", col sep=comma] {Data/kuz_160.txt};
  		\addplot[ultra thick,color0,smooth,each nth point=1] table [x="Points:0", y="f_14-0", col sep=comma] {Data/ras_160.txt};
	\end{axis}
\end{tikzpicture}
\begin{tikzpicture}[scale=0.8]
	\definecolor{color0}{HTML}{4E79A7}
	\definecolor{color1}{HTML}{F28E2B}
	\definecolor{color2}{HTML}{E15759}
	\definecolor{color3}{HTML}{76B7B2}
	\definecolor{color4}{HTML}{59A14F}
	\definecolor{color5}{HTML}{EDC948}
	\definecolor{color6}{HTML}{B07AA1}
	\definecolor{color7}{HTML}{FF9DA7}
	\definecolor{color8}{HTML}{9C755F}
	\definecolor{color9}{HTML}{BAB0AC}
	\begin{axis}[
		height=.3\textheight,
        width=.46\textwidth,
		xmin=0.1, xmax=0.3,
		ymin=0, ymax=1.5e5,
		ytick={0,0.5e5,1e5,1.5e5},
		xtick={0.1,0.2,0.3},
        scaled y ticks=false,
        yticklabel=\empty
		]
  		\addplot[ultra thick,color2,smooth,each nth point=1] table [x="Points:0", y="f_14-0", col sep=comma] {Data/wes_240.txt};
		\addplot[ultra thick,color1,smooth,each nth point=1] table [x="Points:0", y="f_14-0", col sep=comma] {Data/kuz_240.txt};
  		\addplot[ultra thick,color0,smooth,each nth point=1] table [x="Points:0", y="f_14-0", col sep=comma] {Data/ras_240.txt};
	\end{axis}
\end{tikzpicture}
\hspace*{1.75cm}
    \caption{Comparison of simulated enthalpy profiles $h(t)$ for the Westervelt, Kuznetsov and Rasmussen equations. The images correspond to  time $t \in \{1, 1.5,2,3\} \cdot 10^{-5}$ (from top left to bottom right).}
    \label{Fig:WaveEquations}
\end{figure}
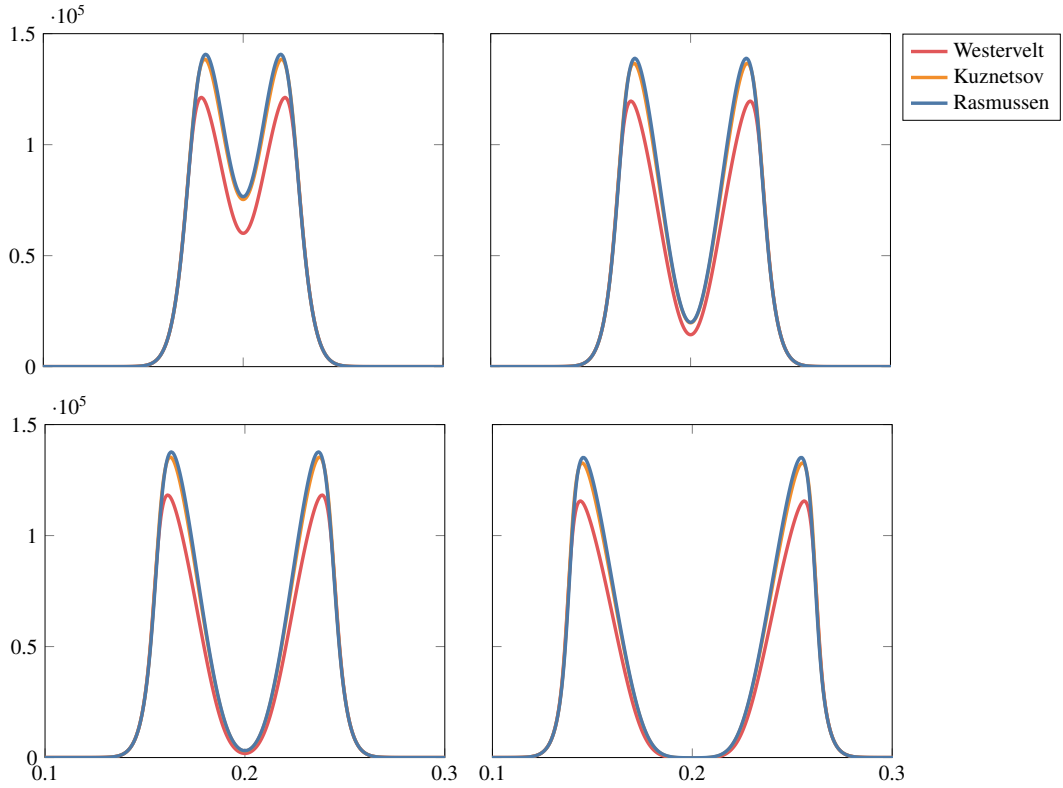
In \cref{Fig:WaveEquations} we depict some snapshots of the enthalpy $h(t)$ obtained by simulation of the Westervelt, Kuznetsov, and Rasmussen equations, which we use to conduct a comparative study of the different models.
A first fundamental observation is that the essential support of the numerical solutions is consistent across all three equations, which illustrated that the waves propagate at the same speed. This was achieved by making an appropriate choice of the parameters. 
However, despite the consistent support, the wave profiles differ significantly between the three models. In particular, the Westervelt equation produces a wave profile with a steeper spike, particularly in the direction of wave propagation. In contrast, the Kuznetsov and Rasmussen equations yield wave profiles that are more centered, with a smoother distribution of wave intensity. 
We further note that the Kuznetsov and Rasmussen equations yield almost identical wave profiles, with only a minor difference in the maximal enthalpy value. 
This indicates that similar accuracy as for the Kuznetsov equations can be obtained with the Rasmussen model in practice. However, the latter additionally guarantees passivity. 
In addition, we plot the evolution of the Hamiltonian energy in \cref{Fig:Energy}. 
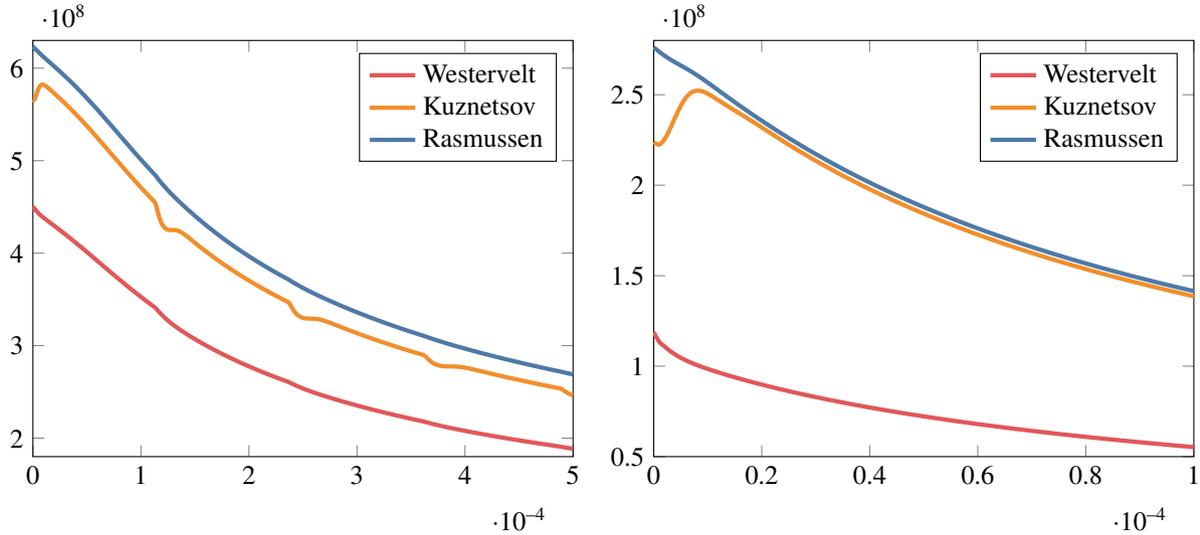
\begin{figure}[htb!]
\centering
    \begin{tikzpicture}
	\definecolor{color0}{HTML}{4E79A7}
	\definecolor{color1}{HTML}{F28E2B}
	\definecolor{color2}{HTML}{E15759}
	\definecolor{color3}{HTML}{76B7B2}
	\definecolor{color4}{HTML}{59A14F}
	\definecolor{color5}{HTML}{EDC948}
	\definecolor{color6}{HTML}{B07AA1}
	\definecolor{color7}{HTML}{FF9DA7}
	\definecolor{color8}{HTML}{9C755F}
	\definecolor{color9}{HTML}{BAB0AC}
	\begin{axis}[
		height=.3\textheight,
        width=.49\textwidth,
		legend pos=north east,
		legend cell align={left},
        restrict y to domain=1.8e8:6.5e8,
        ymin=1.8e8, ymax=6.3e8,
		xmin=0, xmax=0.0005,
		]
		\addplot[ultra thick,color2,smooth,each nth point=2] table [x="t", y="energy", col sep=space]{Data/wesenergy.txt};
		\addlegendentry{Westervelt}
  \addplot[ultra thick,color1,smooth,each nth point=2] table [x="t", y="energy", col sep=space]{Data/kuzenergy.txt};
		\addlegendentry{Kuznetsov}
  \addplot[ultra thick,color0,smooth,each nth point=2] table [x="t", y="energy", col sep=space]{Data/rasenergy.txt};
		\addlegendentry{Rasmussen}
	\end{axis}
\end{tikzpicture}
    \begin{tikzpicture}
	\definecolor{color0}{HTML}{4E79A7}
	\definecolor{color1}{HTML}{F28E2B}
	\definecolor{color2}{HTML}{E15759}
	\definecolor{color3}{HTML}{76B7B2}
	\definecolor{color4}{HTML}{59A14F}
	\definecolor{color5}{HTML}{EDC948}
	\definecolor{color6}{HTML}{B07AA1}
	\definecolor{color7}{HTML}{FF9DA7}
	\definecolor{color8}{HTML}{9C755F}
	\definecolor{color9}{HTML}{BAB0AC}
	\begin{axis}[
		height=.3\textheight,
        width=.49\textwidth,
        legend style={fill=white, fill opacity=0.6, draw opacity=1,text opacity=1},
		legend pos=north east,
		legend cell align={left},
        ymin=.5e8, ymax=2.8e8,
		xmin=0, xmax=0.0001,
		]
		\addplot[ultra thick,color2,smooth,each nth point=2] table [x="t", y="energy", col sep=space]{2D/wesenergy2.txt};
		\addlegendentry{Westervelt}
  \addplot[ultra thick,color1,smooth,each nth point=2] table [x="t", y="energy", col sep=space]{2D/kuzenergy2.txt};
		\addlegendentry{Kuznetsov}
 \addplot[ultra thick,color0,smooth,each nth point=2] table [x="t", y="energy", col sep=space]{2D/rasenergy2.txt};
		\addlegendentry{Rasmussen}
	\end{axis}
\end{tikzpicture}
    \caption{Comparison of the Hamiltonian $\mathcal{H}(v,h)$ for the Westervelt, Kuznetsov and Rasmussen equation in the one-dimensional (left) and two-dimensional example (right).}
    \label{Fig:Energy}
\end{figure}
The energy values differ for each model due to variations in the parameter $a$. We observe that the energies exhibit an exponential decay rate. However, in the case of the Kuznetsov equation, the Hamiltonian is not a true Hamiltonian of the system since it does not satisfy
$c=2d$, a necessary condition for the energy's decay as noted in \cref{pro:weak}. Interestingly, the energy initially increases and shows bumps at approximately $t \in \{1.2,2.5,3.8\} \cdot 10^{-4}$. After these bumps, the energy shows a slight increase.

\subsection{A two-dimensional example}
Next, we consider a two-dimensional example in the domain $\Omega=(0,0.4)^2$. 
We set $b=6 \cdot 10^{-9}$, $c_0=1500$, and $\gamma=1$. The initial data for the two-dimensional problem are chosen as
\begin{align*}
h_0=1.2\cdot 10^6 \cdot \exp\!\bigg(\!\!-\frac{(x_1-0.2)^2+(x_2-0.2)^2}{2\cdot 0.01^2}\bigg).
\end{align*}
We compare the evolution of the total specific enthalpy $h$ for the linear damped wave equation ($a=c=d=0$) and the Westervelt, Kuznetsov, and Rasmussen equations in \cref{Fig:2D}. For all models, the initial signal propagates in a ring-like form from the center towards the boundary. The peak in the linear damped wave equation evolves slower than the nonlinear acoustic equations, i.e., the diameter of the ring is somewhat smaller. 
The speed of propagation for the nonlinear models is very similar. While the solutions for the Kuznetsov and Rasmussen equations are nearly indistinguishable, the peak amplitude of the solution to the Westervelt equation is somewhat smaller than for the other models, and hence displayed lighter. This observation is consistent with our findings in the one-dimensional case, as shown in \cref{Fig:WaveEquations}. 
Similar to the one-dimensional example, see~\cref{Fig:Energy}, the energy for the Kuznetsov equation exhibits an initial increase, wheile the Westervelt, Rasmussen, and linear wave equation show the expected monotone decay behavior illustrating their strict passivity.

\newlength{\imagewidth}
\settowidth{\imagewidth}{\includegraphics{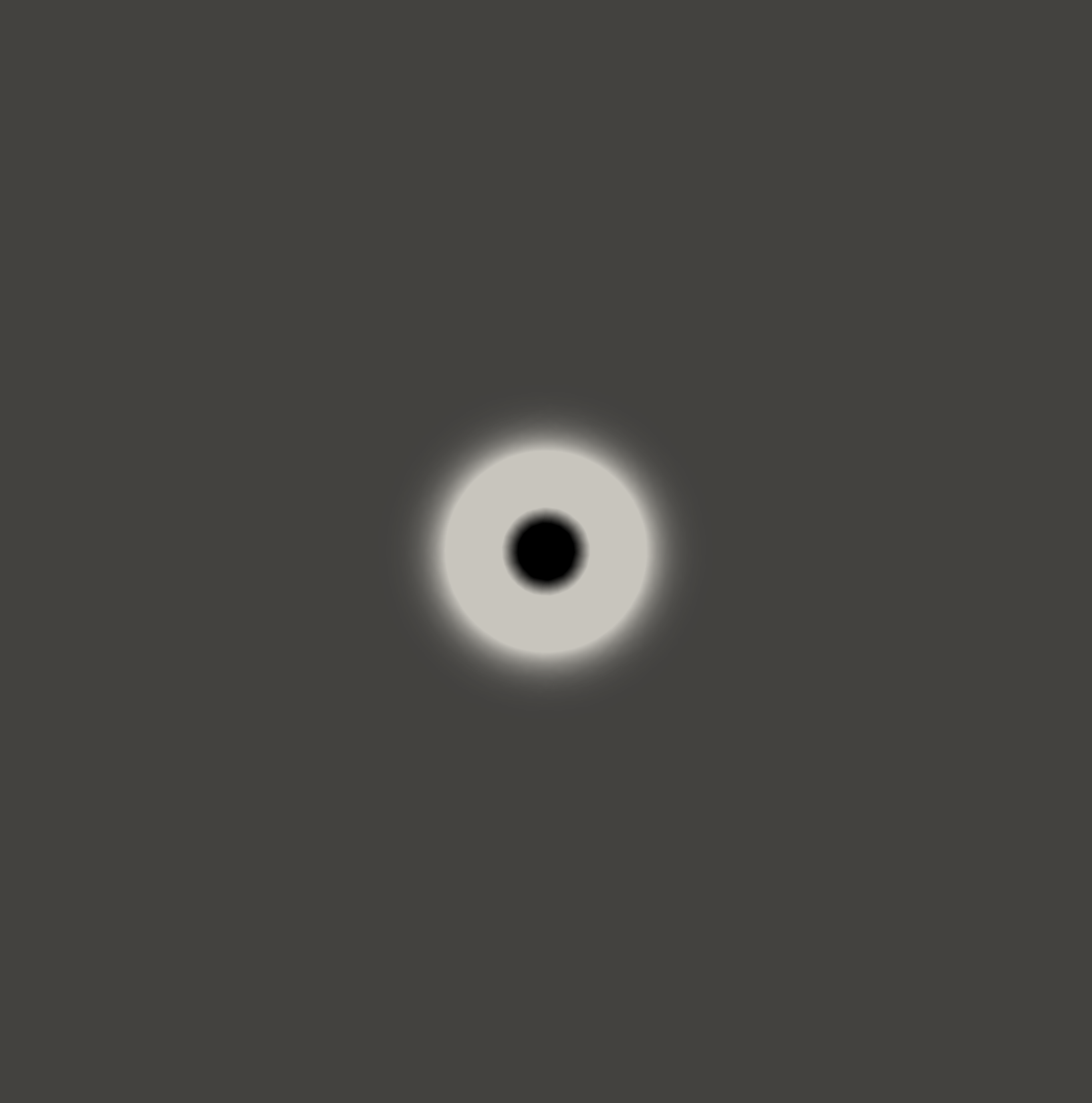}}

\begin{figure}[htb!]
\centering
\begin{minipage}{.3\textwidth}
\includegraphics[trim=0 0.5\imagewidth{} 0.5\imagewidth{} 0, clip, width = .49\textwidth]{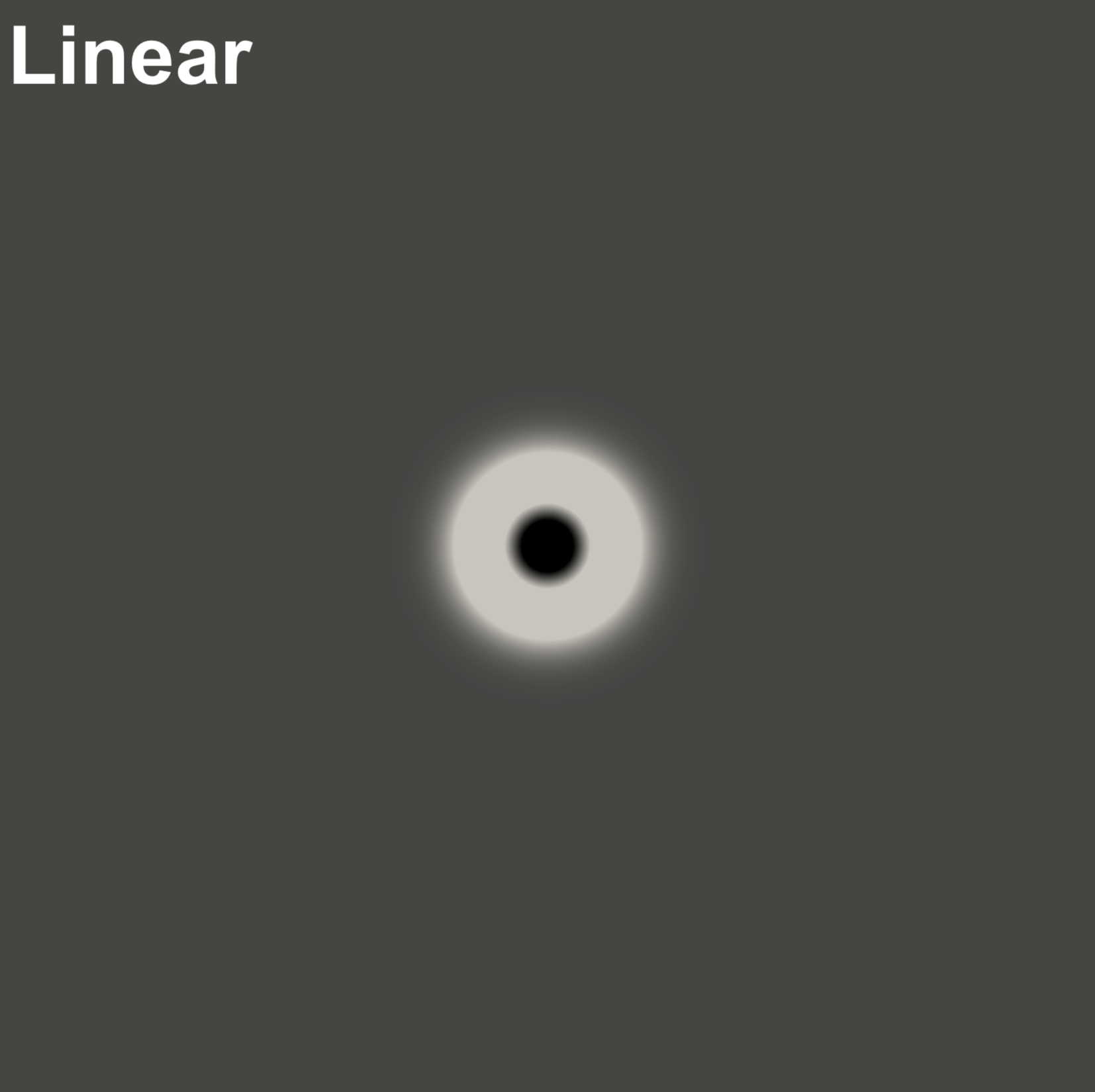}\hspace{-.08cm}
\includegraphics[trim=0.5\imagewidth{} 0.5\imagewidth{} 0 0, clip, width = .49\textwidth]{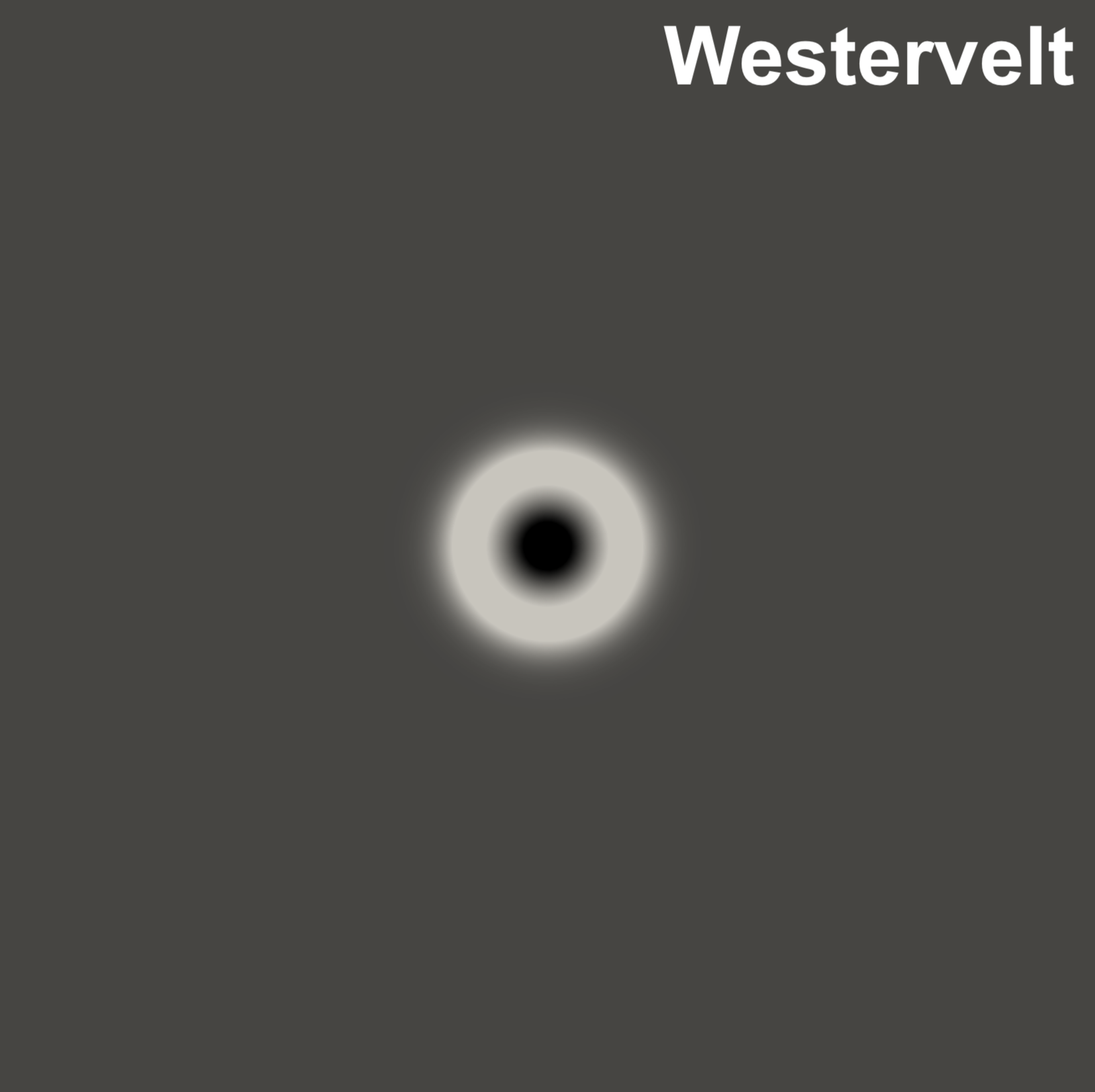}\\[.01cm]
\includegraphics[trim=0 0 0.5\imagewidth{} 0.5\imagewidth{}, clip, width = .49\textwidth]{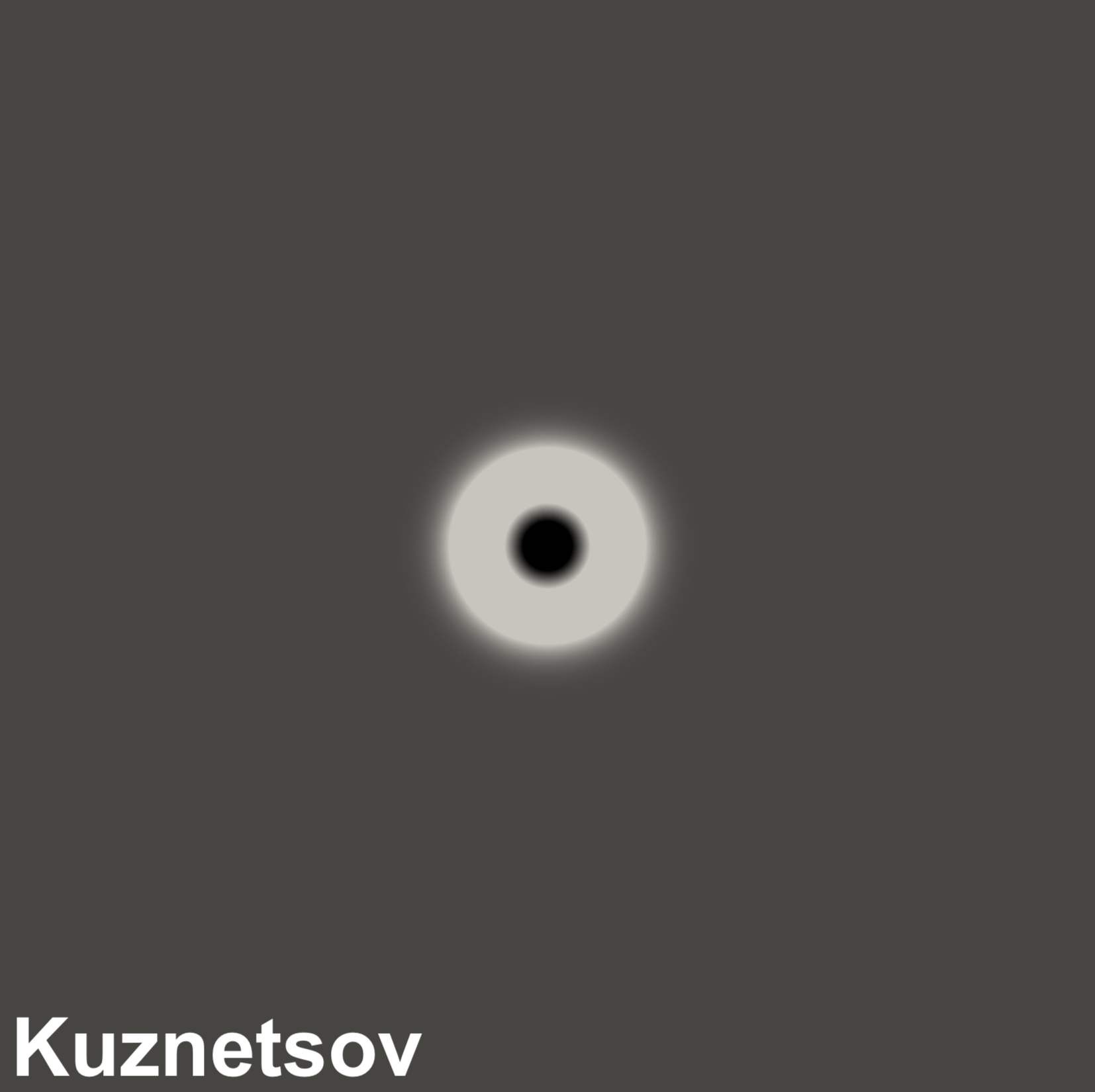}\hspace{-.08cm}
\includegraphics[trim=0.5\imagewidth{} 0 0 0.5\imagewidth{}, clip, width = .49\textwidth]{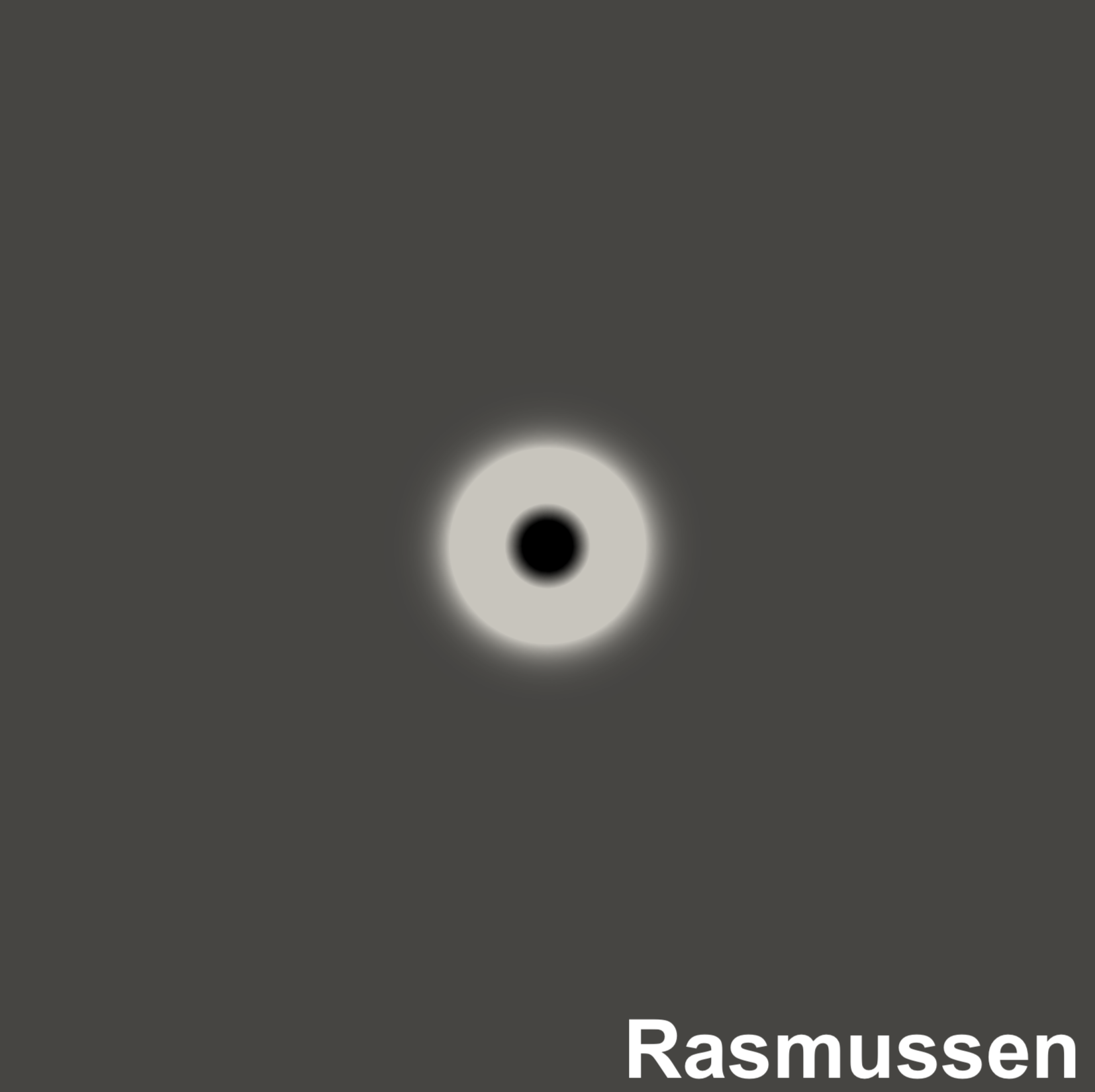}
\end{minipage}\hspace{.8cm}
\begin{minipage}{.3\textwidth}
\includegraphics[trim=0 0.5\imagewidth{} 0.5\imagewidth{} 0, clip, width = .49\textwidth]{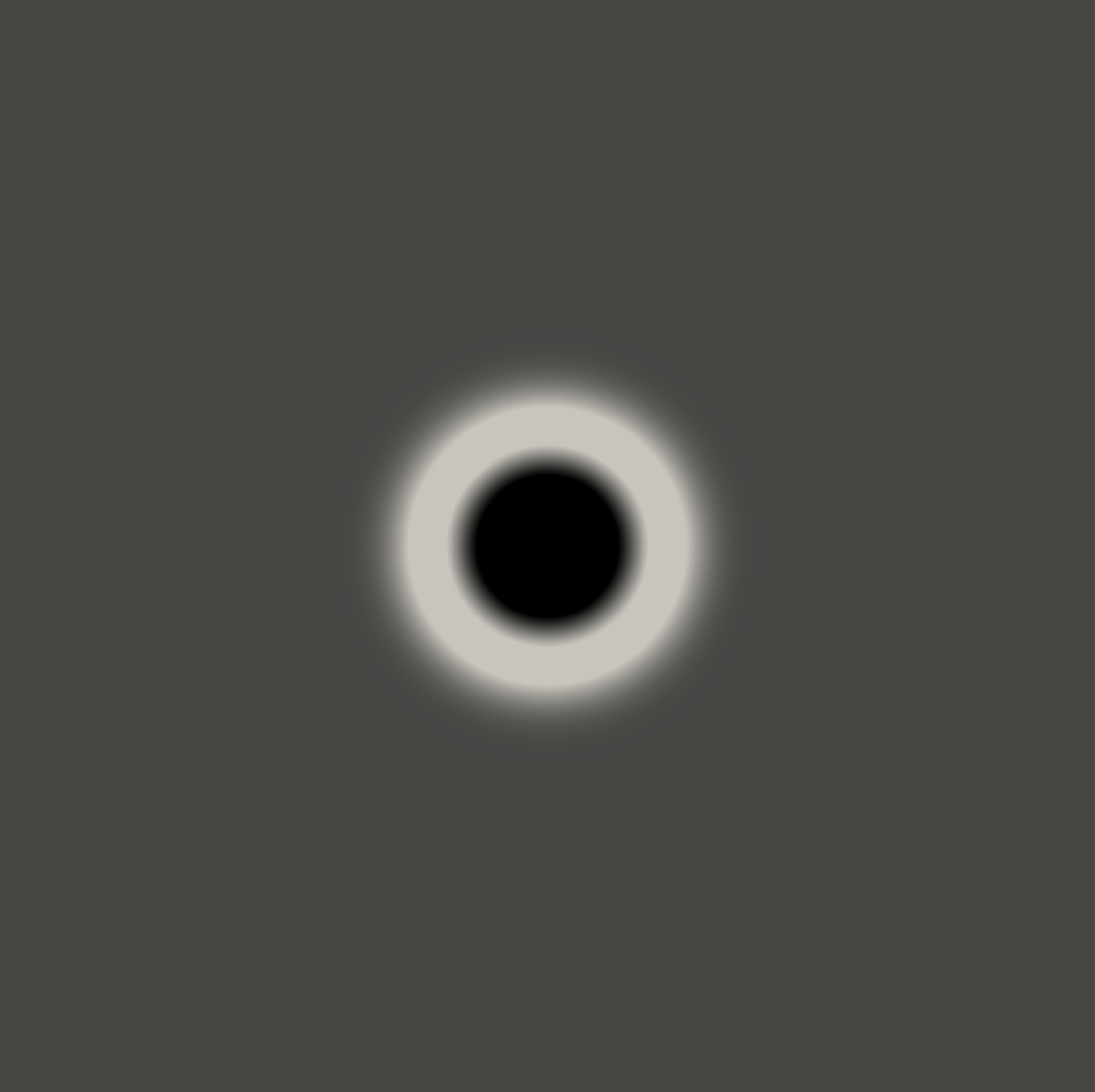}\hspace{-.1cm}
\includegraphics[trim=0.5\imagewidth{} 0.5\imagewidth{} 0 0, clip, width = .49\textwidth]{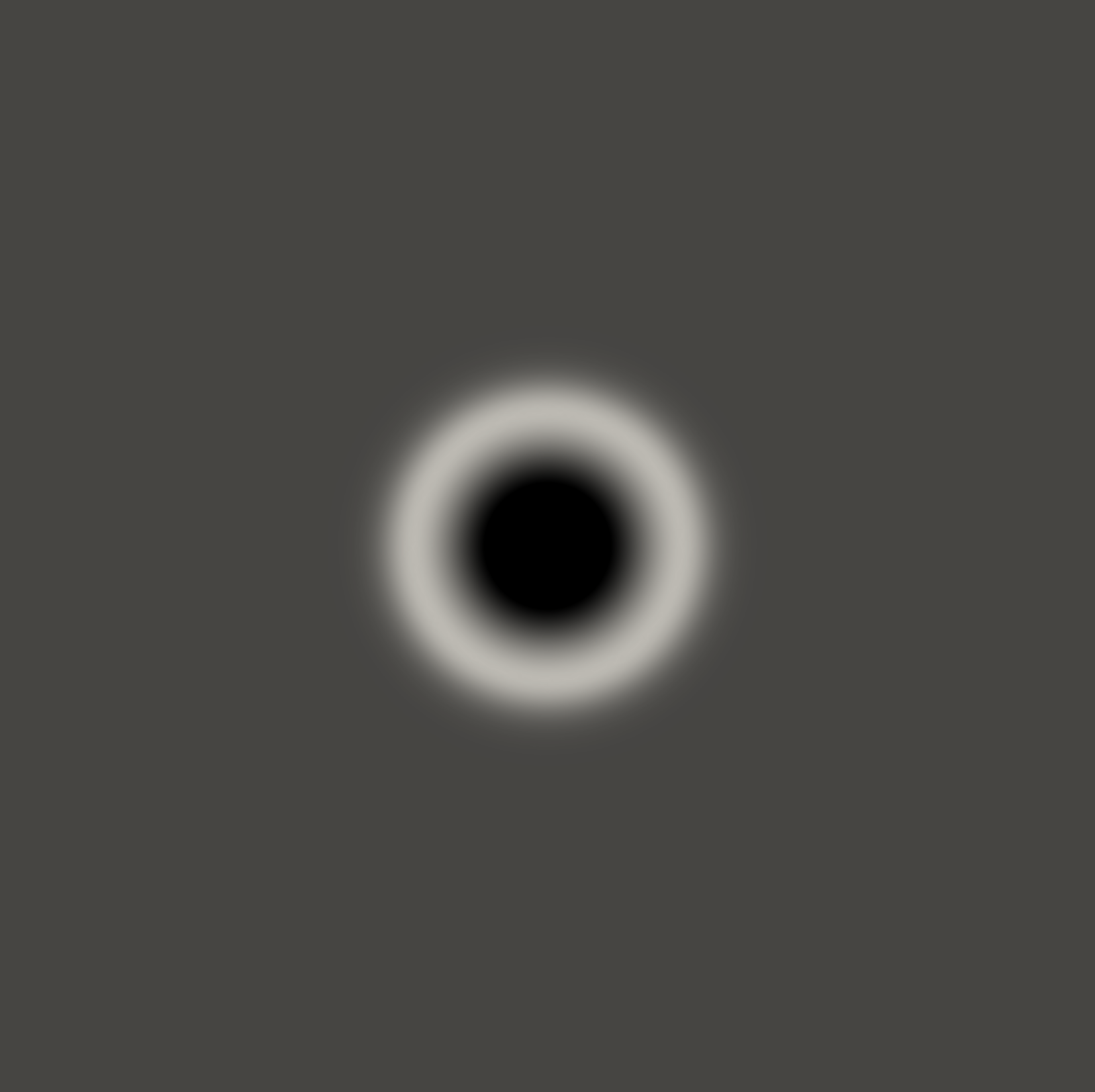}\\[0cm]
\includegraphics[trim=0 0 0.5\imagewidth{} 0.5\imagewidth{}, clip, width = .49\textwidth]{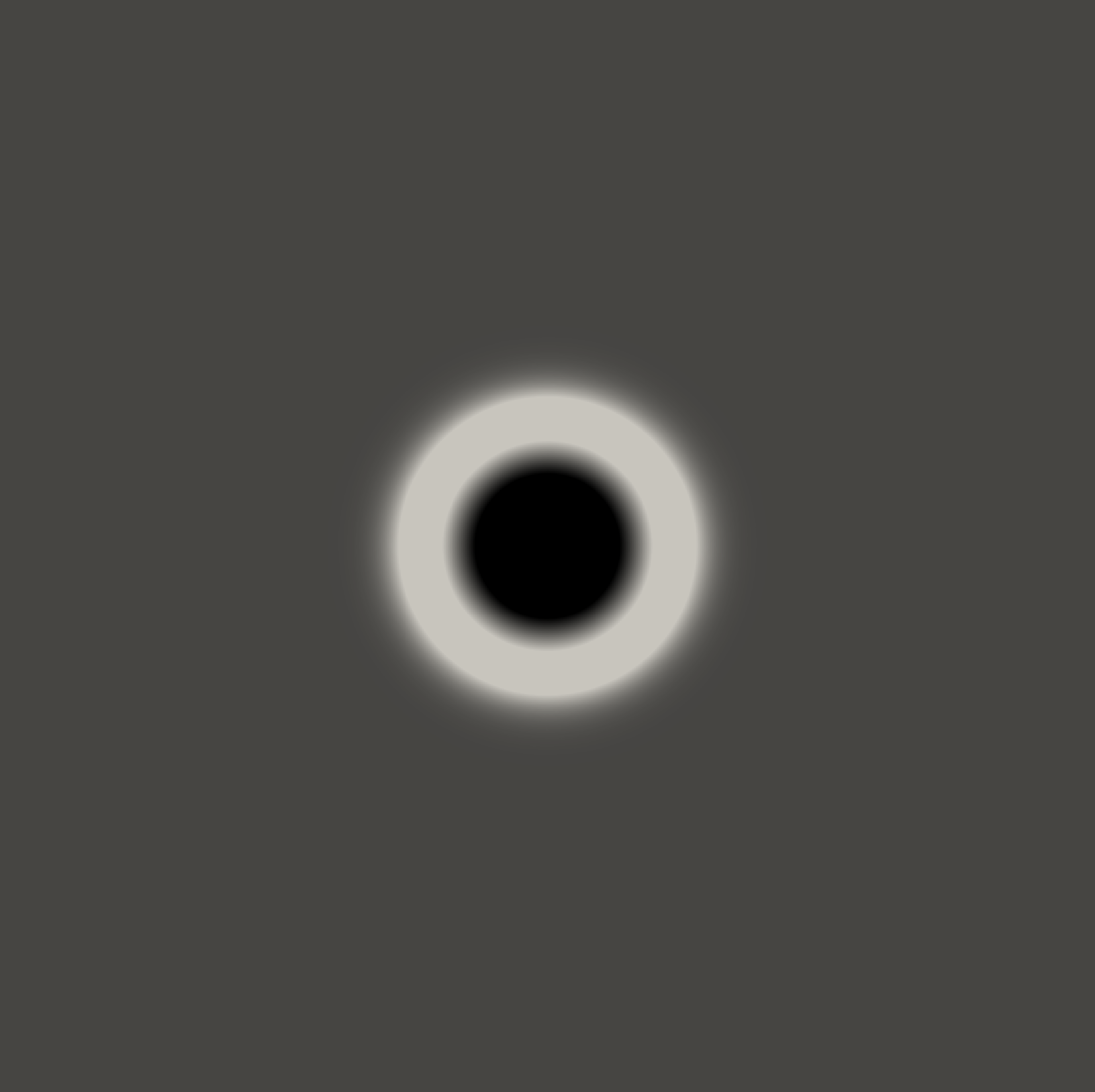}\hspace{-.1cm}
\includegraphics[trim=0.5\imagewidth{} 0 0 0.5\imagewidth{}, clip, width = .49\textwidth]{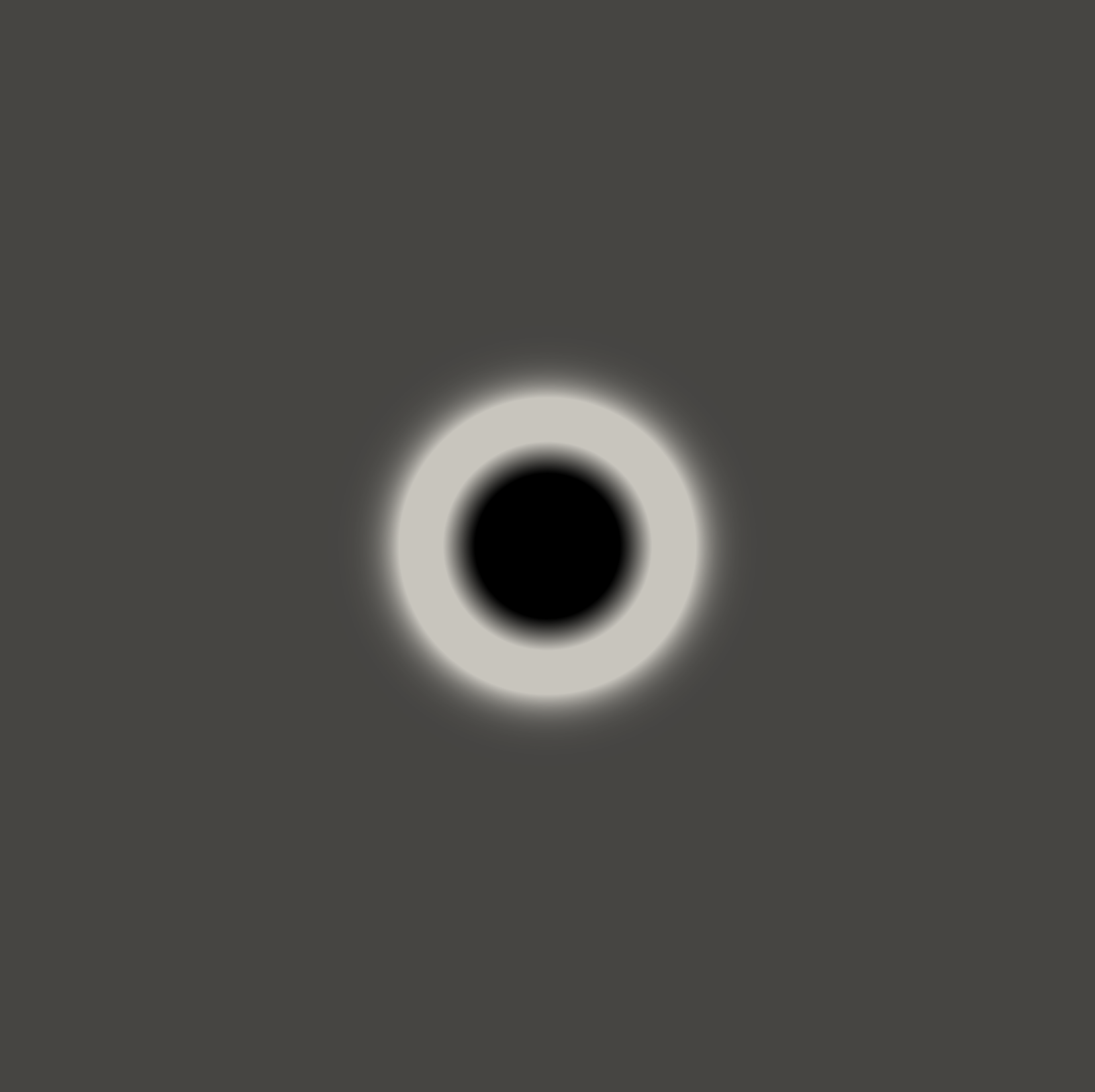}
\end{minipage}\hspace{.8cm} 
\hphantom{\begin{minipage}{.11\textwidth}
\includegraphics[width=.99\textwidth]{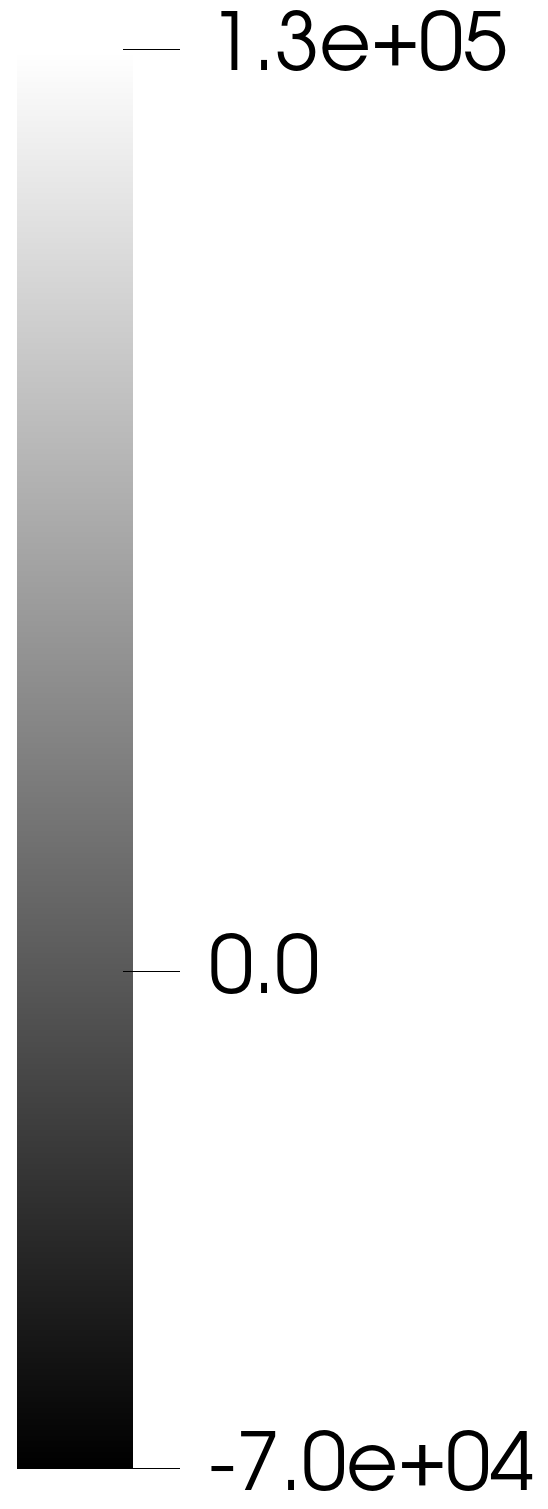}
\end{minipage}}
\\[.8cm] 
\begin{minipage}{.3\textwidth}
\includegraphics[trim=0 0.5\imagewidth{} 0.5\imagewidth{} 0, clip, width = .49\textwidth]{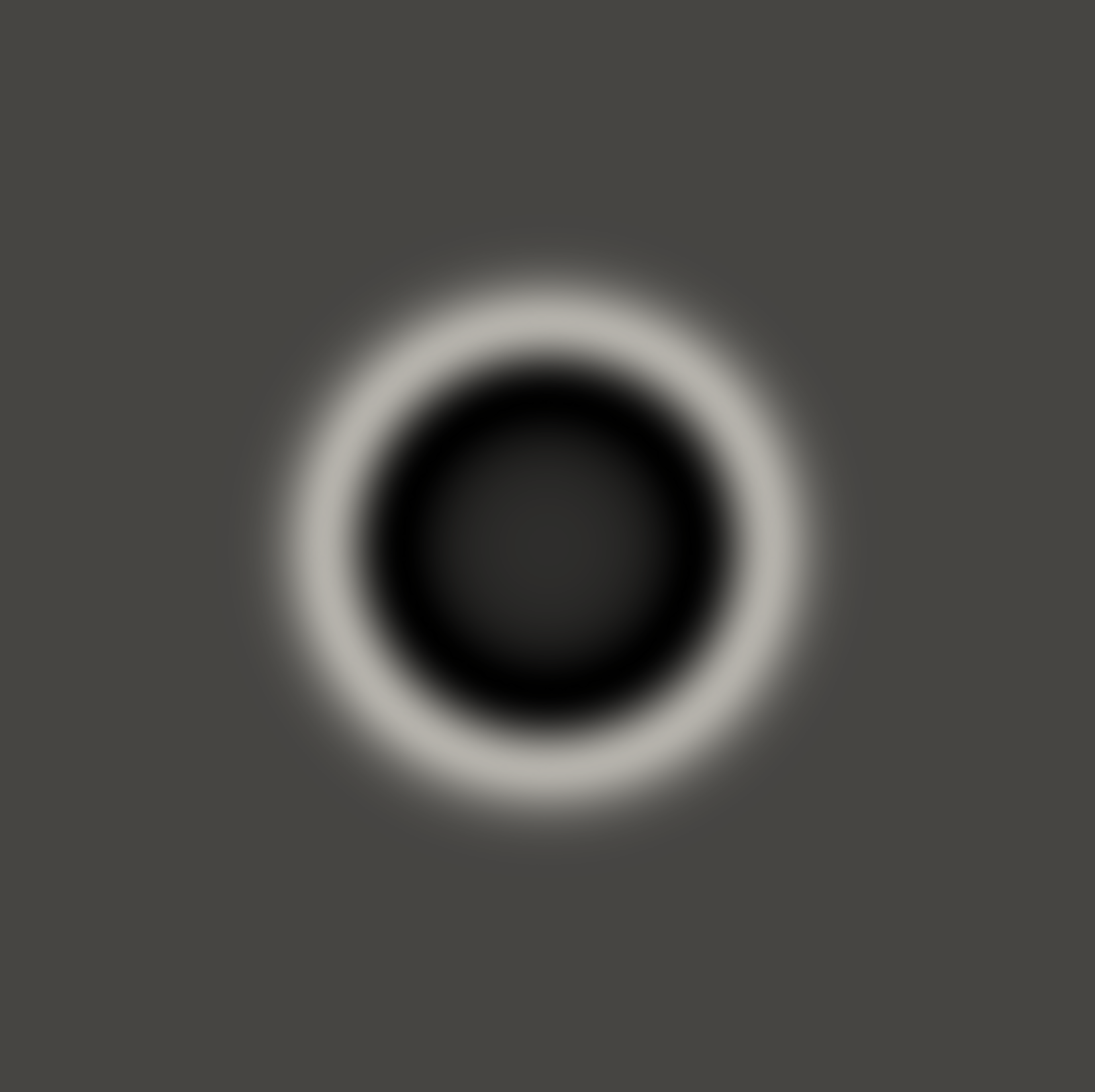}\hspace{-.1cm}
\includegraphics[trim=0.5\imagewidth{} 0.5\imagewidth{} 0 0, clip, width = .49\textwidth]{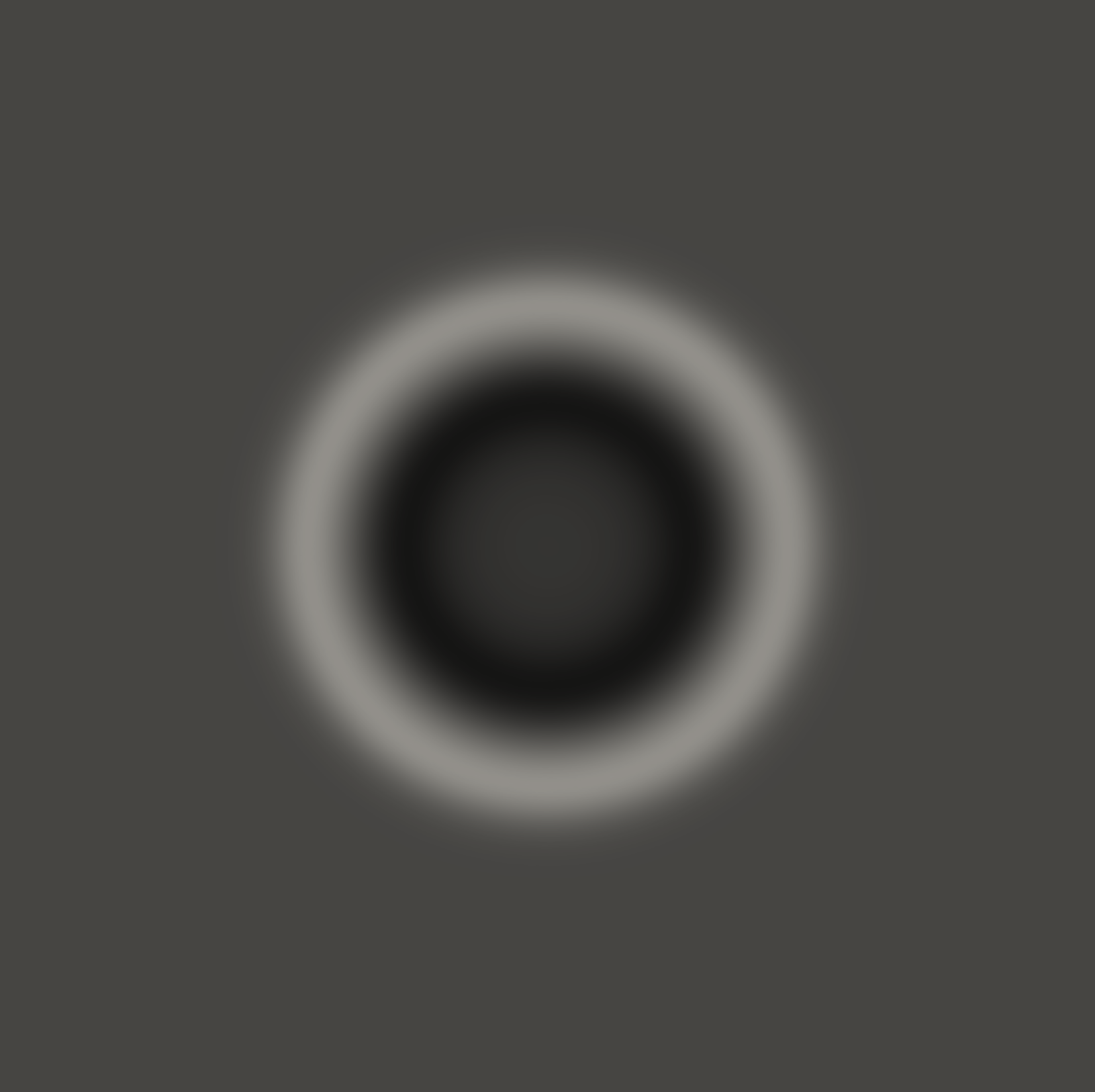}\\[0cm]
\includegraphics[trim=0 0 0.5\imagewidth{} 0.5\imagewidth{}, clip, width = .49\textwidth]{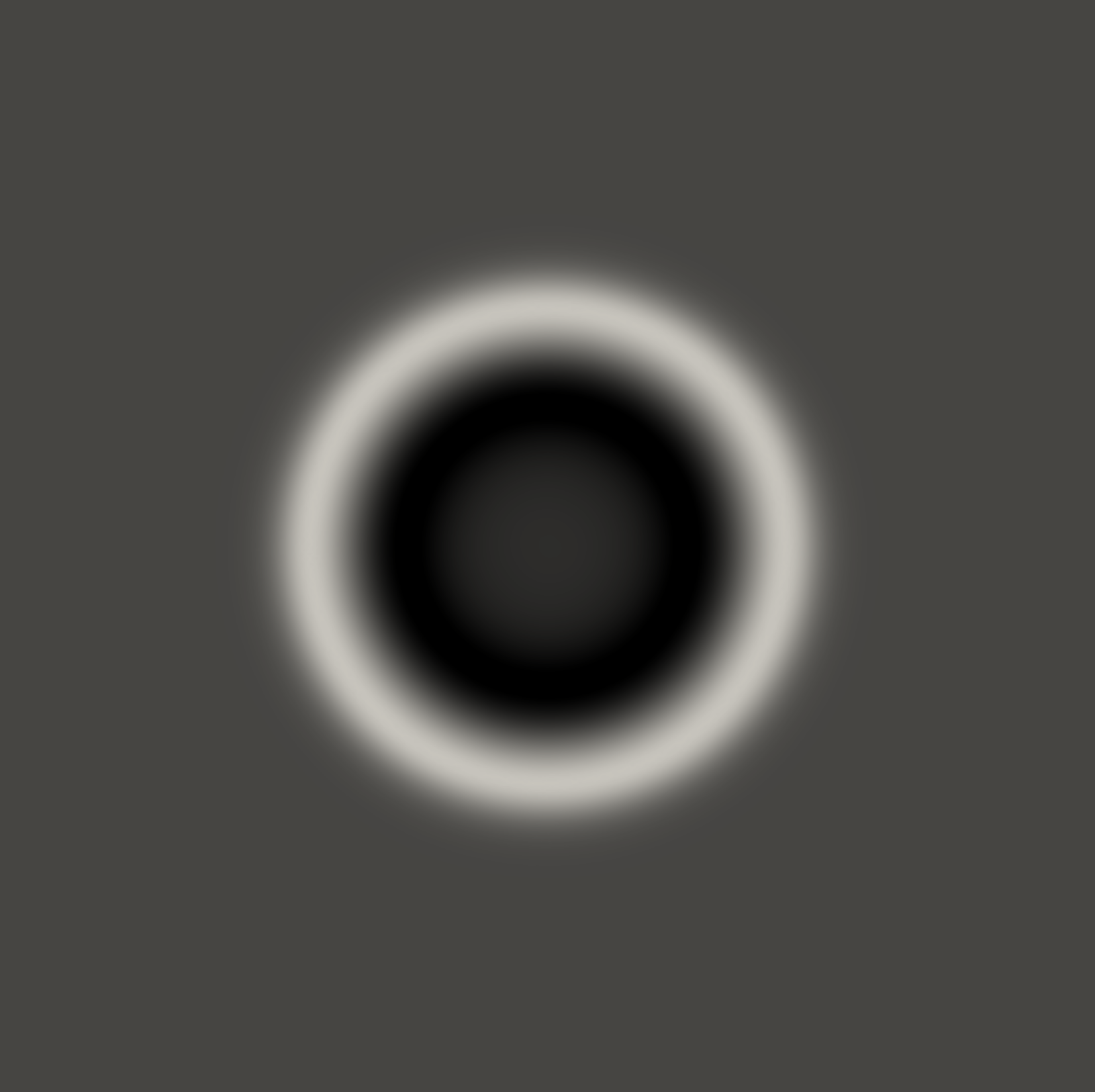}\hspace{-.1cm}
\includegraphics[trim=0.5\imagewidth{} 0 0 0.5\imagewidth{}, clip, width = .49\textwidth]{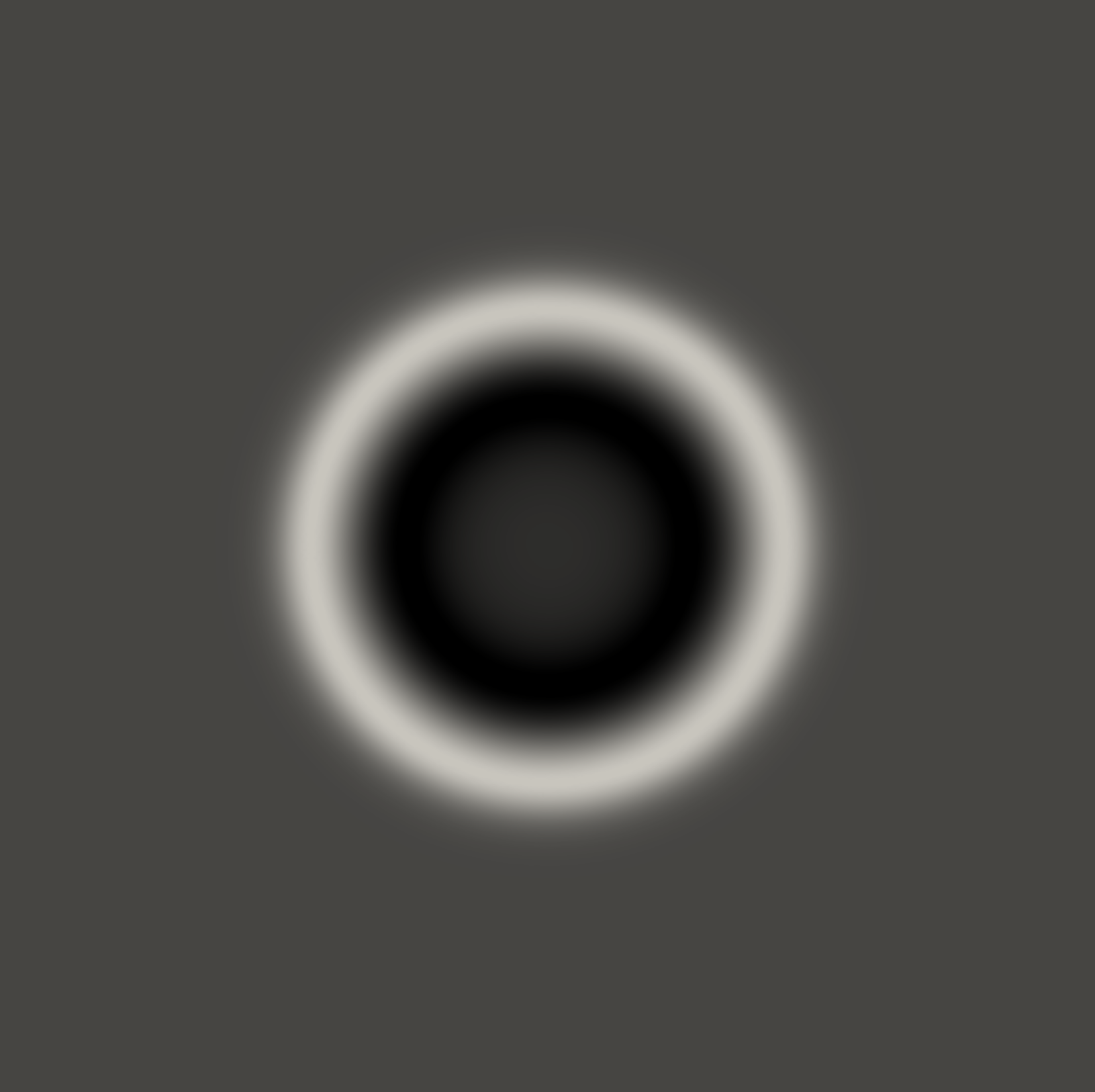}
\end{minipage}\hspace{.8cm}
\begin{minipage}{.3\textwidth}
\includegraphics[trim=0 0.5\imagewidth{} 0.5\imagewidth{} 0, clip, width = .49\textwidth]{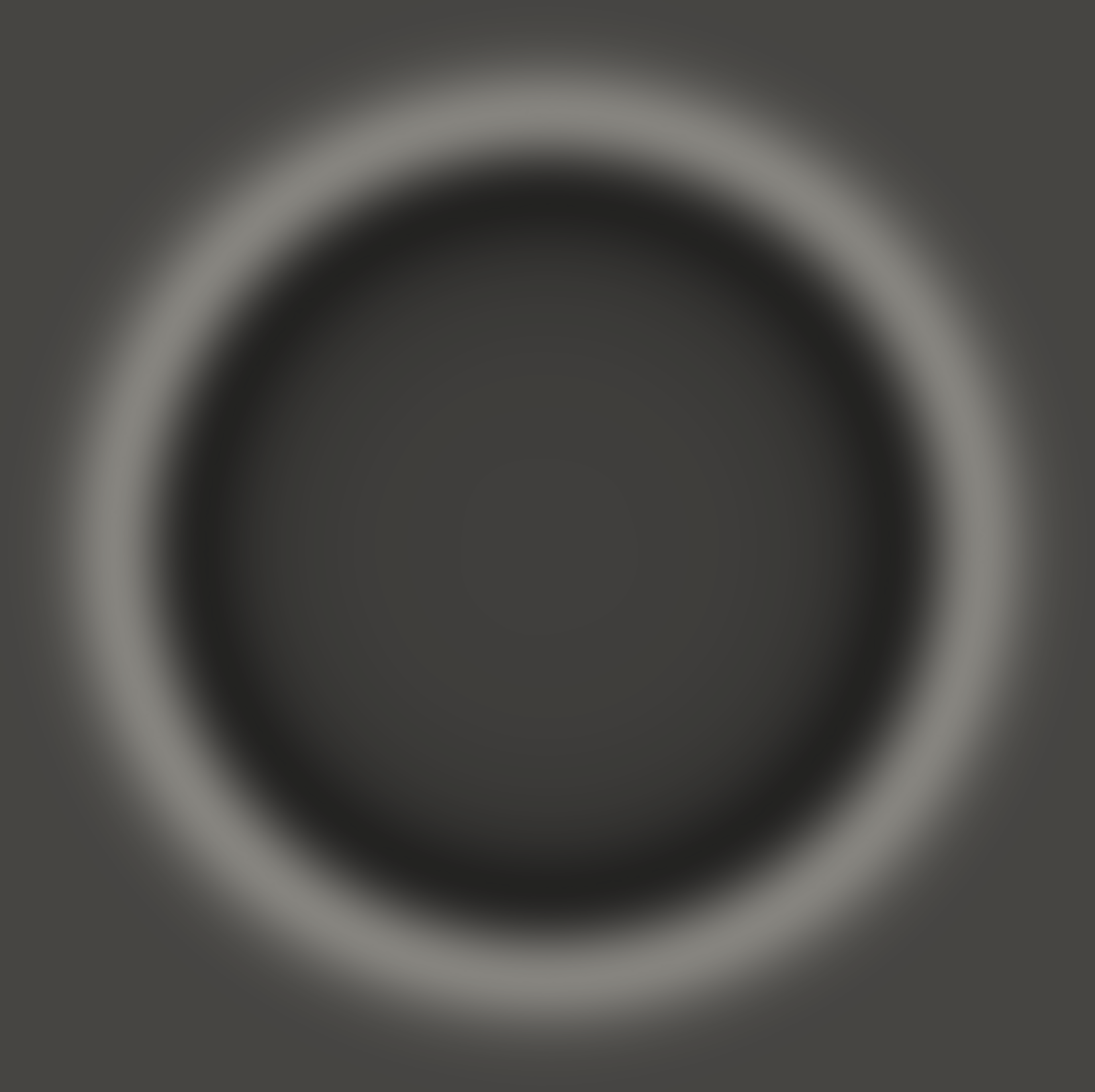}\hspace{-.1cm}
\includegraphics[trim=0.5\imagewidth{} 0.5\imagewidth{} 0 0, clip, width = .49\textwidth]{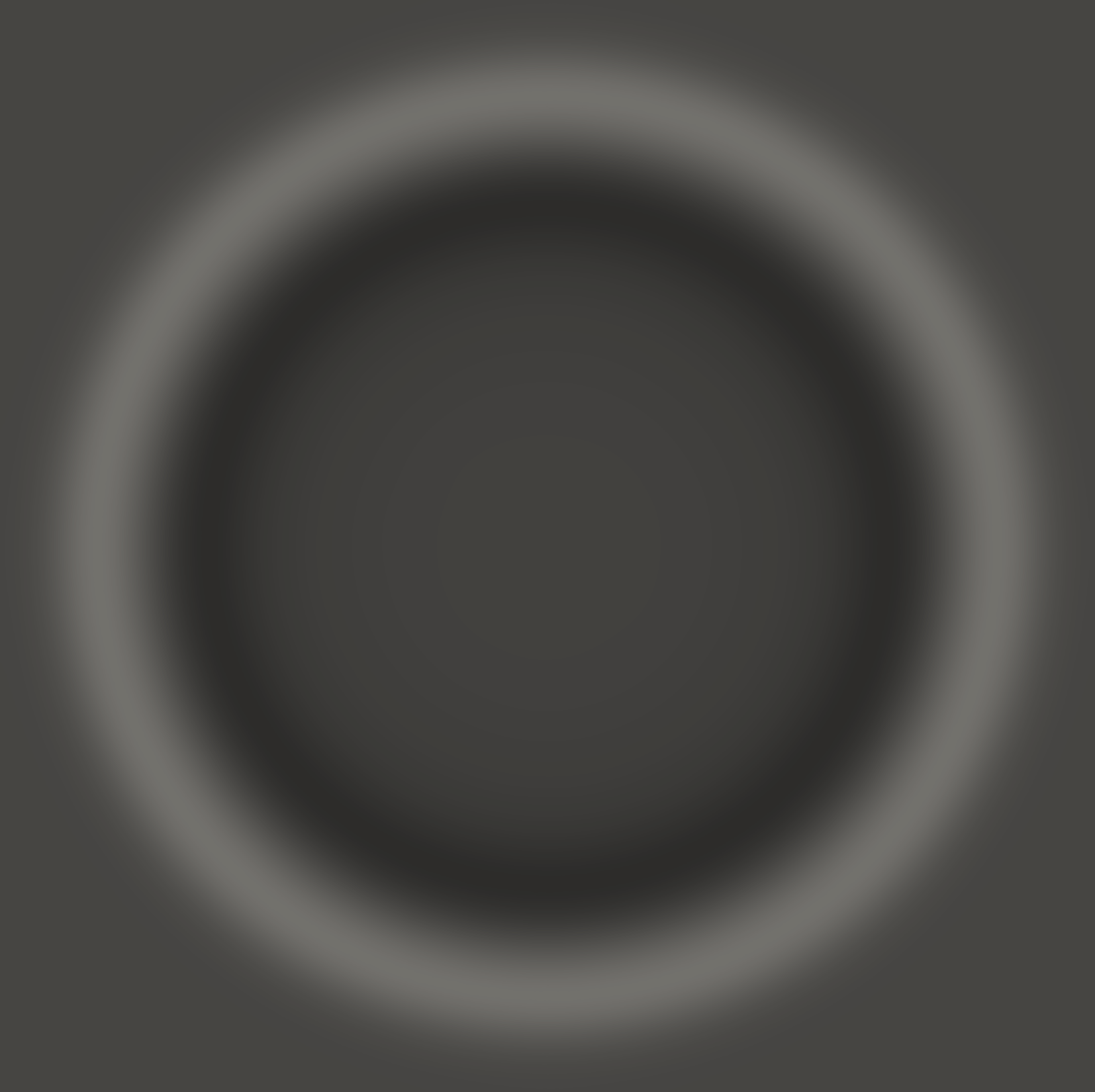}\\[0cm]
\includegraphics[trim=0 0 0.5\imagewidth{} 0.5\imagewidth{}, clip, width = .49\textwidth]{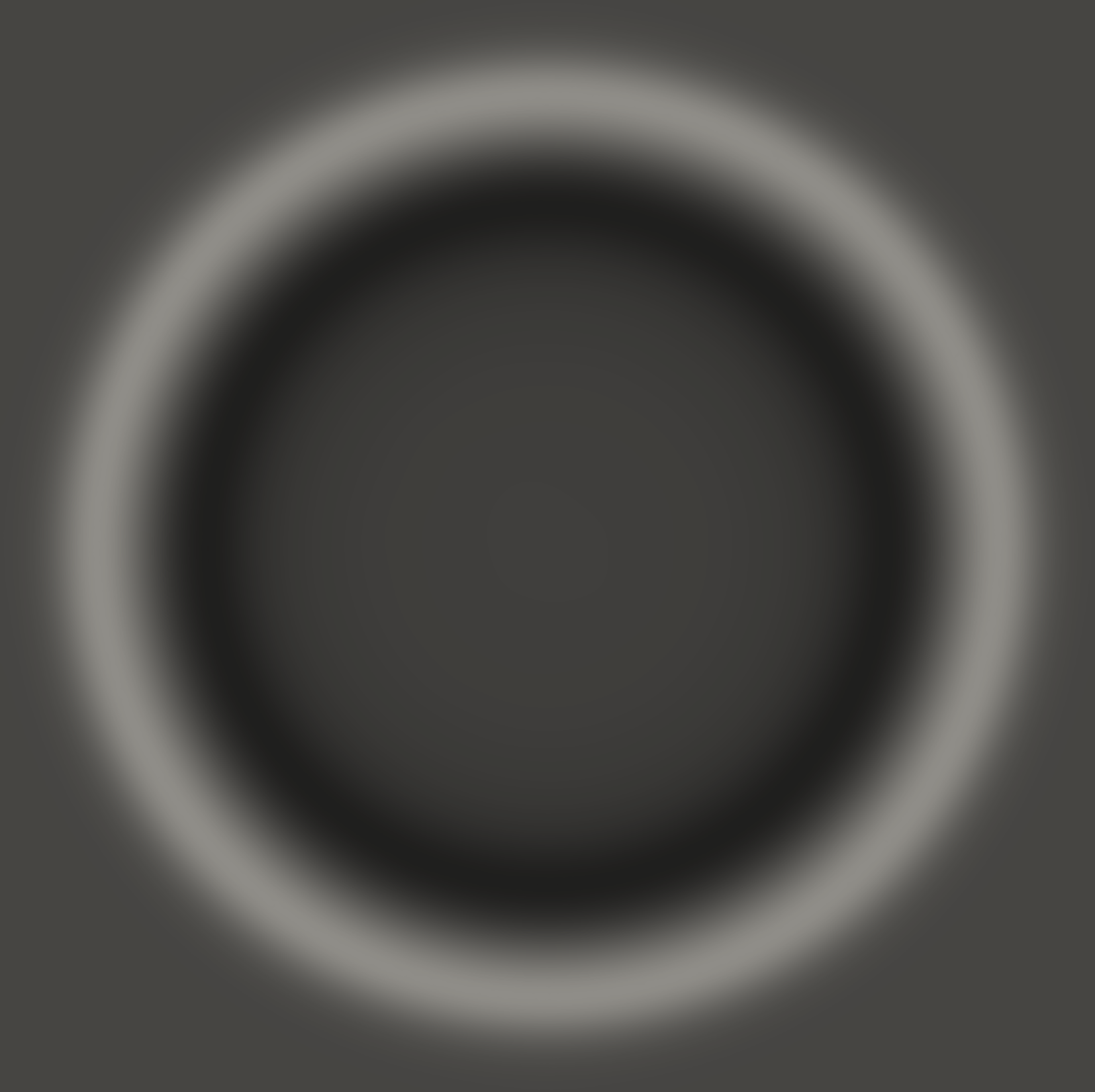}\hspace{-.1cm}
\includegraphics[trim=0.5\imagewidth{} 0 0 0.5\imagewidth{}, clip, width = .49\textwidth]{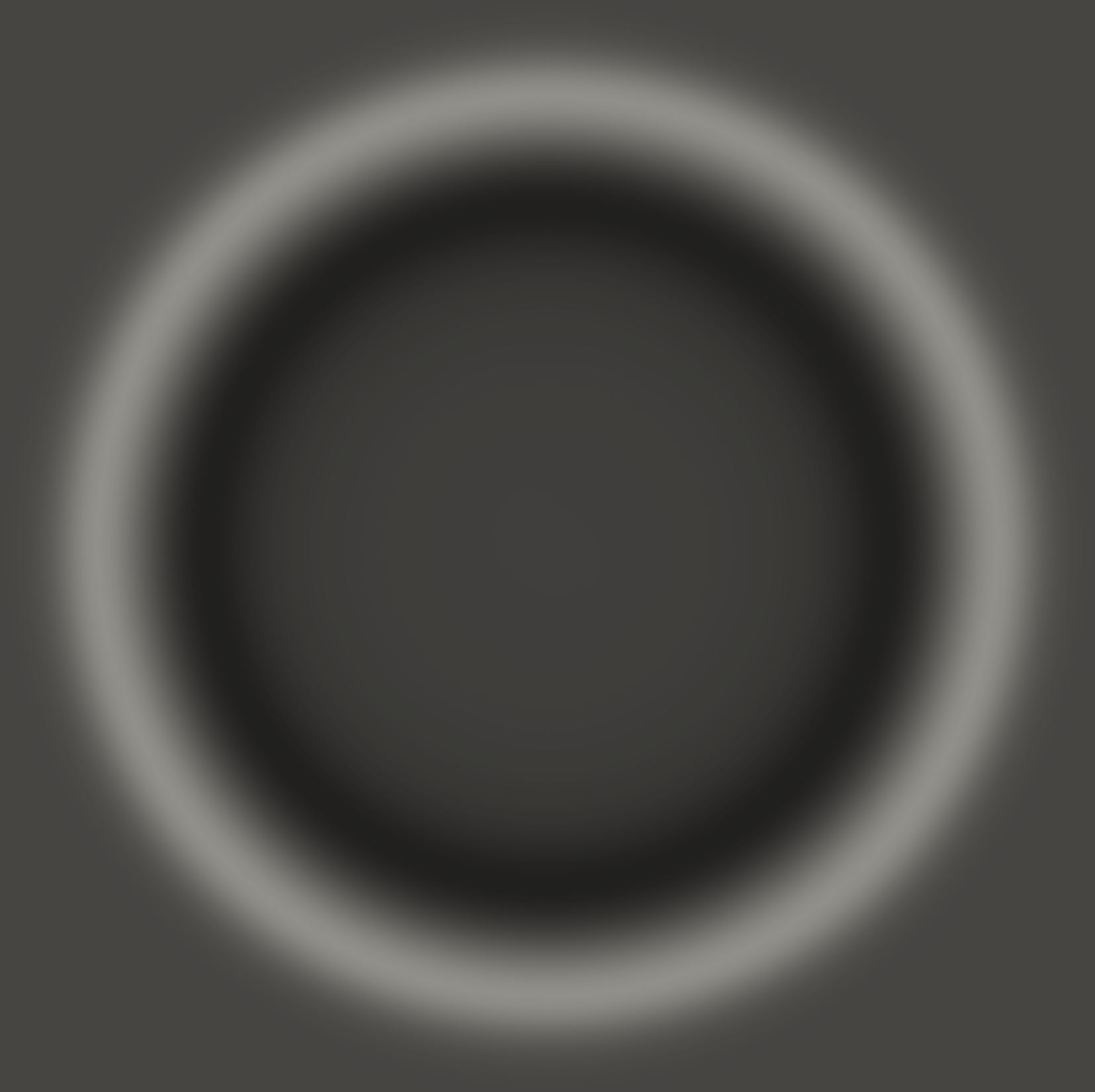}
\end{minipage}\hspace{.8cm} 
\begin{minipage}{.11\textwidth}
\includegraphics[width=.99\textwidth]{2D/color2.png}
\end{minipage}
\caption{Linear (upper left) vs Westervelt (upper right) vs Kuznetsov (lower left) vs Rasmussen (lower right) for $t \in \{2.5,5,10,20\}\cdot 10^{-5}$.\label{Fig:2D}}
\end{figure}

\section{Conclusion and outlook}

In this paper, we analyzed a class of models for nonlinear acoustics, including the Westervelt and Kuznetsov equations, as well as a thermodynamically consistent modification of the latter proposed by Rasmussen. 
By employing linearization, energy estimates, and fixed-point arguments, we have demonstrated the existence and uniqueness of solutions for these models.
Our results show that, for sufficiently small initial data, the solutions are global in time and converge exponentially fast to equilibrium. 
While the Kuznetsov and Rasmussen models require similar implementation effort and yield very similar results in simulations, the latter enjoys thermodynamic consistency, allowing to guarantee energy-dissipation on the continuous and discrete level. This key property is directly encoded in the port-Hamiltonian structure of the weak form of the velocity-enthalpy formulation. 

We showed that the geometric structure, encoded in the weak form of the problem, can be preserved by conforming Galerkin approximation, e.g., by appropriate mixed finite elements. We further indicated that fully--discrete structure-preserving discretization schemes can be obtained by variational time-discretization schemes. 
Energy-decay was, however, also observed for simulations based on the implicit midpoint rule. For one- and two-dimensional problems, the fully-implicit time-stepping schemes turned out to be very stable and computationally effective. The Newton method converged consistently within a few iterations, demonstrating both the efficiency and reliability of the overall discretization approach.

In closing this section, we mention some directions for future research. 
A natural extension of the models would be to include dispersive effects, which are critical to accurately capturing the behavior of acoustic waves in various media.
Apart from establishing mathematical well-posedness, the guarantee of thermodynamic consistency should be a particular focus. 
Another direction would be a thorough analysis of discretization schemes for \eqref{eq:sys1}--\eqref{eq:sys4}, including variational time-discretization and a full error analysis. 
Establishing $L^\infty$ bounds for the discrete solutions and studying the long-time behavior would be further topics of interest. 

\section*{Acknowledgement}
M.F.~gratefully acknowledges the support of the State of Upper Austria. 

\bibliography{literature.bib}

\end{document}